\documentclass[11pt,reqno]{amsart}

\usepackage{a4wide}
\usepackage[english]{babel}
\usepackage[utf8]{inputenc}
\usepackage{amsmath}
\usepackage{braket}
\usepackage[hidelinks=true]{hyperref}
\usepackage{relsize}
\usepackage{tikz}
\usepackage{here}

\newtheorem{thm}{Theorem}[section]
\newtheorem{prob}[thm]{Problem}
\newtheorem{lem}[thm]{Lemma}

\newtheorem{cor}[thm]{Corollary}

\numberwithin{equation}{section}

\newcommand{\dom}{\operatorname{dom}}
\newcommand{\Sc}{\operatorname{Sc}}
\renewcommand{\Im}{\operatorname{Im}}

\title[Spectral properties of the gradient operator]{Spectral properties of the gradient operator with nonconstant coefficients}

\author[Fabrizio Colombo]{F. Colombo}
\address{(FC) Politecnico di Milano, Dipartimento di Matematica, Via E. Bonardi 9, 20133 Milano, Italy}
\email{fabrizio.colombo@polimi.it}

\author[Francesco Mantovani]{F. Mantovani}
\address{(FM) Politecnico di Milano, Dipartimento di Matematica, Via E. Bonardi 9, 20133 Milano, Italy}
\email{francesco.mantovani@polimi.it}

\author[Peter Schlosser]{P. Schlosser}
\address{(PS) Politecnico di Milano, Dipartimento di Matematica, Via E. Bonardi 9, 20133 Milano, Italy}
\email{pschlosser@math.tugraz.at}

\begin{document}

\begin{abstract}
In mathematical physics, the gradient operator with nonconstant coefficients encompasses various models, including Fourier's law for heat propagation and Fick's first law, that relates the diffusive flux to the gradient of the concentration. Specifically, consider $n\geq 3$ orthogonal unit vectors $e_1,\dots,e_n\in\mathbb{R}^n$, and let $\Omega\subseteq\mathbb{R}^n$ be some (in general unbounded) Lipschitz domain. This paper investigates the spectral properties of the gradient operator $T=\sum_{i=1}^ne_ia_i(x)\frac{\partial}{\partial x_i}$ with nonconstant positive coefficients $a_i:\overline{\Omega}\to(0,\infty)$. Under certain regularity and growth conditions on the $a_i$, we identify bisectorial or strip-type regions that belong to the $S$-resolvent set of $T$. Moreover, we obtain suitable estimates of the associated resolvent operator. Our focus lies in the spectral theory on the $S$-spectrum, designed to study the operators acting in Clifford modules $V$ over the Clifford algebra $\mathbb{R}_n$, with vector operators being a specific crucial subclass. The spectral properties related to the $S$-spectrum of $T$ are linked to the inversion of the operator $Q_s(T):=T^2-2s_0T+|s|^2$, where $s\in\mathbb{R}^{n+1}$ is a paravector, i.e., it is of the form $s=s_0+s_1e_1+\dots+s_ne_n$. This spectral problem is substantially different from the complex one, since it allows to associate general boundary conditions to $Q_s(T)$, i.e., to the squared operator $T^2$.
\end{abstract}

\maketitle

AMS Classification  47A10, 47A60. \medskip

 \medskip

\textbf{Acknowledgements:} Fabrizio Colombo is supported by MUR grant Dipartimento di Eccellenza 2023-2027. Peter Schlosser was funded by the Austrian Science Fund (FWF) under Grant No. J 4685-N and by the European Union--NextGenerationEU.

\section{Introduction}

Quaternionic and Clifford operators encompass vector operators, which arise in numerous scientific contexts. Contrary to complex holomorphic function theory, there are different possible notions of hyperholomorphicity  in the noncommutative setting of $\mathbb{R}_n$. Noncommutative spectral theories then rely on various concepts of spectra, based on distinct Cauchy kernels depending on the type of hyperholomorphicity being considered. \medskip

The spectral theory, grounded in the notion of $S$-spectrum and associated with slice hyperholomorphicity, initiated its development in 2006, showing that this noncommutative framework is not merely an extension deducible from complex operator theory, especially when dealing with unbounded operators, see \cite{RD}. The identification of the $S$-spectrum emerged through methodologies in hypercomplex analysis, although its roots were initially hinted by quaternionic quantum mechanics, see \cite{BF}. An exhaustive description on the discovery of the $S$-spectrum can be found in the introduction of the book \cite{CGK}. The spectral theory on the $S$-spectrum was developed in several papers and it is systematically organized in the books \cite{ACS2016,AlpayColSab2020,FJBOOK,CGK,ColomboSabadiniStruppa2011}, where one can find also the development of the theory of slice hyperholomorphic functions. It is also important to point out that the quaternionic and the Clifford spectral theorems are based on the notion of $S$-spectrum as well, see \cite{ACK,ColKim}. However, this spectral theory goes beyond the quaternionic and the Clifford setting, as shown in \cite{ADVCGKS}. \medskip

Let us give a short explanation of the differences between complex and $S$-spectral theory. The resolvent operator of a complex linear operator $A$ is directly linked to the definition of its spectrum via the invertibility of the operator $\lambda-A$. In particular, if $A$ is bounded, the inverse of this operator admits the series representation
\begin{equation*}
\sum_{n=0}^\infty A^n\lambda^{-1-n}=(\lambda-A)^{-1},\qquad|\lambda|>\Vert A\Vert.
\end{equation*}
However, this identity is not true for Clifford operators $T$. In this case there rather holds
\begin{subequations}
\begin{align}
\sum_{n=0}^\infty T^ns^{-1-n}=(T^2-2s_0T+|s|^2)^{-1} (\overline{s}-T),\qquad|s|>\Vert T\Vert, \label{Eq_SL_series} \\
\sum_{n=0}^\infty s^{-1-n}T^n=(\overline{s}-T)(T^2-2s_0T+|s|^2)^{-1},\qquad|s|>\Vert T\Vert, \label{Eq_SR_series}
\end{align}
\end{subequations}
for $s\in\mathbb{R}_n$ of the form $s=s_0+s_1e_1+\dots+s_ne_n$. This relation suggests, that in contrast to the complex spectral theory, the definition of the $S$-spectrum is connected to the invertibility of the second order operator
\begin{equation}\label{Eq_Qs}
Q_s(T):=T^2-2s_0T+|s|^2.
\end{equation}
That is, for every $T\in\mathcal{B}(V)$, the \textit{$S$-spectrum} and the \textit{$S$-resolvent set} of $T$ are defined as
\begin{equation*}
\rho_S(T):=\Set{s\in\mathbb{R}^{n+1} | Q_s(T)^{-1}\in\mathcal{B}(V)}\qquad\text{and}\qquad\sigma_S(T):=\mathbb{R}^{n+1}\setminus\rho_S(T).
\end{equation*}
Here $V$ is a Clifford module and $\mathcal{B}(V)$ the set of all bounded right-linear operators $T:V\to V$. For every $s\in\rho_S(T)$, the inverse $Q_s(T)^{-1}$ is called the \textit{pseudo $S$-resolvent operator}. \medskip

A second main difference between complex and $S$-spectral theory, lies in the fact that
\begin{equation*}
\rho_S(T)\ni s\mapsto Q_s(T)^{-1},
\end{equation*}
is not a (operator valued) slice hyperholomorphic function. The operator $Q_s(T)$ is only used for the definition of the $S$-spectrum, while the \textit{left} and the \textit{right $S$-resolvent operators} are motivated by the series representations \eqref{Eq_SL_series} and \eqref{Eq_SR_series}
\begin{equation*}
S_L^{-1}(s,T):=Q_s(T)^{-1}(\overline{s}-T)\qquad\text{and}\qquad S_R^{-1}(s,T):=(\overline{s}-T)Q_s(T)^{-1},\qquad s\in\rho_S(T).
\end{equation*}
These resolvent operators now preserve the slice hyperholomorphicity and are used in the definition of the $S$-functional calculus for Clifford operators. \medskip

Clifford operators that are sectorial, bi-sectorial or strip-operators, can be defined in analogy to complex operators, but in this case the estimates are associated with the $S$-resolvents
\begin{equation*}
\Vert S_L^{-1}(s,T)\Vert\leq\frac{C}{|s|}\qquad\text{and}\qquad\Vert S_R^{-1}(s,T)\Vert\leq\frac{C}{|s|},
\end{equation*}
where the parameter $s$ belongs to suitable sectorial, bi-sectorial or strip-type subsets of $\rho_S(T)$. \medskip

In this paper we will study the spectral properties of the
first order differential operator
\begin{equation}\label{Eq_T}
T=\sum_{i=1}^ne_ia_i(x)\frac{\partial}{\partial x_i},\qquad x\in\Omega,
\end{equation}
with Dirichlet, Neumann and Robin boundary conditions associated to $Q_s(T)$. Here $e_1,\dots,e_n$ are the imaginary units of the Clifford algebra $\mathbb{R}_n$ and the coefficients $a_1,\dots,a_n$ are defined on some minimally smooth domain $\Omega\subseteq\mathbb{R}^n$ in the sense of \cite[Section VI.3.3]{S70}, in $n\geq 3$ dimensions. The restriction to $n\geq 3$ is necessary since the crucial Sobolev inequality in Lemma~\ref{lem_Sobolev_inequality} only holds true in those dimensions. For the coefficients of $T$ in \eqref{Eq_T} we require that $a_1,\dots,a_n:\overline{\Omega}\mapsto(0,\infty)$ belong to $C^1(\overline{\Omega})$ and satisfy the bounds
\begin{subequations}\label{Eq_Coefficient_bounds}
\begin{align}
m_a:=&\min\limits_{i\in\{1,\dots,n\}}\inf\limits_{x\in\Omega}a_i(x)>0, \label{Eq_ma} \\
M_a:=&\Big(\sum\nolimits_{i=1}^n\Vert a_i\Vert_\infty^2\Big)^{\frac{1}{2}}<\infty, \label{Eq_Ma} \\
M_a':=&\Big(\sum\nolimits_{i,j=1}^n\Big\Vert a_j\frac{\partial a_i}{\partial x_j}\Big\Vert^2_{L^n}\Big)^{\frac{1}{2}}<\infty, \label{Eq_Maprime}
\end{align}
\end{subequations}
where $\Vert\cdot\Vert_{L^n}$ is the usual Lebesgue norm of real valued functions and $\Vert\cdot\Vert_\infty$ is the supremum norm of continuous functions. The Clifford module in which we will be working is
\begin{equation*}
L^2(\Omega)=\Set{f:\Omega\to\mathbb{R}_n | \int_\Omega|u(x)|^2dx<\infty},
\end{equation*}
that is the space of square integrable functions with values in the Clifford algebra $\mathbb{R}_n$. \medskip

In this paper we will interpret the invertibility of the operator $Q_s(T)$ in \eqref{Eq_Qs} in the weak sense. That is, we ask whether the partial differential equation
\begin{equation}\label{Eq_Qs_spectral_problem}
Q_s(T)u(x)=(T^2-2s_0T+|s|^2)u(x)=f(x),\qquad x\in\Omega,
\end{equation}
equipped with boundary conditions of Dirichlet, Neumann or Robin type, admits a unique weak solution for every $f\in L^2(\Omega)$. In order to derive this weak formulation, we first note that the square of the operator $T$ in \eqref{Eq_T} is given by
\begin{align}
T^2&=\sum_{i,j=1}^ne_ja_j\frac{\partial}{\partial x_j}\Big(e_ia_i\frac{\partial}{\partial x_i}\Big) \notag \\
&=-\sum_{i=1}^na_i\frac{\partial}{\partial x_i}\Big(a_i\frac{\partial}{\partial x_i}\Big)+\sum_{i,j=1,j\neq i}^ne_je_ia_j\frac{\partial a_i}{\partial x_j}\frac{\partial}{\partial x_i}+\underbrace{\sum_{i,j=1,j\neq i}^ne_je_ia_ja_i\frac{\partial^2}{\partial x_j\partial x_i}}_{\makebox[0pt]{=0}} \notag \\[-0.6cm]
&=-\sum_{i=1}^na_i\frac{\partial}{\partial x_i}\Big(a_i\frac{\partial}{\partial x_i}\Big)-\sum_{i=1}^ne_i\Big(B_i-e_ia_i\frac{\partial a_i}{\partial x_i}\Big)\frac{\partial}{\partial x_i}, \label{Eq_T2}
\end{align}
where the third sum in the second line vanishes due to anti-commutation property of the units of the Clifford algebra $e_je_i=-e_ie_j$ for every $j\neq i$, and in the last line we used the abbreviation
\begin{equation}\label{Eq_Bi}
B_i:=\sum_{j=1}^ne_ja_j\frac{\partial a_i}{\partial x_j},\qquad i\in\{1,\dots,n\}.
\end{equation}
Applying this operator to some $u\in H^2(\Omega)$, testing with $v\in H^1(\Omega)$ and using integration by parts with the outer unit normal vector $\vec{\nu}=(\nu_1,\dots,\nu_n)$ of the boundary $\partial\Omega$, gives
\begin{align*}
\langle T^2u,v\rangle_{L^2}&=-\sum_{i=1}^n\Big\langle a_i\frac{\partial}{\partial x_i}\Big(a_i\frac{\partial u}{\partial x_i}\Big),v\Big\rangle_{L^2}-\sum_{i=1}^n\Big\langle e_i\Big(B_i-e_ia_i\frac{\partial a_i}{\partial x_i}\Big)\frac{\partial u}{\partial x_i},v\Big\rangle_{L^2} \\
&=\sum_{i=1}^n\Big\langle a_i\frac{\partial u}{\partial x_i},\frac{\partial}{\partial x_i}(a_iv)\Big\rangle_{L^2}-\sum_{i=1}^n\Big\langle\nu_ia_i\frac{\partial u}{\partial x_i},a_iv\Big\rangle_{L^2(\partial\Omega)} \\
&\quad-\sum_{i=1}^n\Big\langle\frac{\partial u}{\partial x_i},\Big(B_i-e_ia_i\frac{\partial a_i}{\partial x_i}\Big)e_iv\Big\rangle_{L^2} \\
&=\sum_{i=1}^n\Big\langle\frac{\partial u}{\partial x_i},a_i^2\frac{\partial v}{\partial x_i}-B_ie_iv\Big\rangle_{L^2}-\sum_{i=1}^n\Big\langle \nu_i\frac{\partial u}{\partial x_i},a_i^2v\Big\rangle_{L^2(\partial\Omega)}.
\end{align*}
Consequently, the weak formulation of the problem \eqref{Eq_Qs_spectral_problem} is connected to the form
\begin{equation}\label{Eq_qs_formal}
q_s(u,v):=\sum_{i=1}^n\Big\langle\frac{\partial u}{\partial x_i},a_i^2\frac{\partial v}{\partial x_i}+(2s_0a_i-B_i)e_iv\Big\rangle_{L^2}+|s|^2\langle u,v\rangle_{L^2}-\sum_{i=1}^n\Big\langle\nu_i\frac{\partial u}{\partial x_i},a_i^2v\Big\rangle_{L^2(\partial\Omega)}.
\end{equation}
The primary objective of this paper is to delve into this spectral problem with Dirichlet, Neumann and Robin-type boundary conditions, establishing bi-sectorial estimates for the solution and hence the $S$-resolvent operator. These findings are an extension of the quaternionic framework explored in prior publications such as \cite{BARACCO,ColomboDenizPinton2020,ColomboDenizPinton2021,ColomboPelosoPinton2019}, and in some cases contains results never investigated in previous works. \medskip

The structure of the paper is outlined as follows: In Section~\ref{sec_Lax_Milgram}, we present the main results concerning modules over the Clifford algebra $\mathbb{R}_n$. Among those is the Lax-Milgram lemma, which will be crucial for the solvability of the form \eqref{Eq_qs_formal} in Section~\ref{sec_Dirichlet} and Section~\ref{sec_Robin}. Section~\ref{sec_Dirichlet} focuses on the $S$-spectral problem \eqref{Eq_Qs_spectral_problem} with Dirichlet boundary conditions $u=0$ on $\partial\Omega$, and similarly, Section~\ref{sec_Robin} investigates Robin-type boundary conditions $\sum_{i=1}^n\nu_ia_i\frac{\partial u}{\partial x_i}+bu=0$ on $\partial\Omega$. Additional to the unique weak solvability, for both kinds of boundary conditions, $L^2$- and $H^1$-estimates of the solution are established. Consequently, sectoriality estimates of the resolvent operators can be deduced. For the readability of the paper, an auxiliary maximization problem is put in Section \ref{sec_Maximization_problem}. The final Section \ref{sec_Concluding_remarks} contains concluding remarks on the current state of this theory.

\section{Functional setting in Clifford modules over $\mathbb{R}_n$}\label{sec_Lax_Milgram}

In this section we consider for $n\in\mathbb{N}$, $n\geq 3$ the \textit{Clifford algebra} $\mathbb{R}_n$, generated by $n$ \textit{imaginary units} $e_1,\dots,e_n$ which satisfy the relations
\begin{equation*}
e_i^2=-1\qquad\text{and}\qquad e_ie_j=-e_je_i,\qquad i\neq j\in\{1,\dots,n\}.
\end{equation*}
More precisely, $\mathbb{R}_n$ is given by
\begin{equation}
\mathbb{R}_n:=\Set{\sum\nolimits_{A\in\mathcal{A}}x_Ae_A | x_A\in\mathbb{R},\,A\in\mathcal{A}},
\end{equation}
using the index set
\begin{equation*}
\mathcal{A}:=\Set{(i_1,\dots,i_r) | r\in\{0,\dots,n\},\,1\leq i_1<\dots<i_r\leq n},
\end{equation*}
and the \textit{basis vectors} $e_A:=e_{i_1}\dots e_{i_r}$. Note, for $A=\emptyset$ the empty product of imaginary units is the real number $e_\emptyset:=1$. Moreover, we will consider the set of all \textit{paravectors}
\begin{equation*}
\mathbb{R}^{n+1}:=\Set{x_0+\sum\nolimits_{i=1}^nx_ie_i | x_0,x_1,\dots,x_n\in\mathbb{R}}.
\end{equation*}
For any Clifford number $x\in\mathbb{R}_n$, we define
\begin{align*}
\Sc(x)&:=x_\emptyset=x_0, && \textit{(scalar part)} \\
\overline{x}&:=\sum\nolimits_{A\in\mathcal{A}}x_A\overline{e_A}, && \textit{(conjugate)} \\
|x|&:=\Big(\sum\nolimits_{A\in\mathcal{A}}|x_A|^2\Big)^{\frac{1}{2}}=(\Sc(x\overline{x}))^{\frac{1}{2}}=(\Sc(\overline{x}x))^{\frac{1}{2}}, && \textit{(norm)}
\end{align*}
where $\overline{e_A}=\overline{e_{i_r}}\dots\overline{e_{i_1}}$ and $\overline{e_i}=-e_i$. It is now obvious that for any Clifford number $x\in\mathbb{R}_n$ one can calculate its coefficients $x_A$ by
\begin{equation}\label{Eq_xA}
\Sc(x\overline{e_A})=\Sc\Big(\sum\nolimits_{B\in\mathcal{A}}x_Be_B\overline{e_A}\Big)=\sum\nolimits_{B\in\mathcal{A}}x_B\Sc(e_B\overline{e_A})=x_A,
\end{equation}
where in the last equation we used that
\begin{equation}\label{Eq_SceBeA}
\Sc(e_B\overline{e_A})=\begin{cases} 1, & \text{if }B=A, \\ 0, & \text{if }B\neq A. \end{cases}
\end{equation}
For any real Hilbert space $V_\mathbb{R}$ with inner product $\langle\cdot,\cdot\rangle_\mathbb{R}$ and norm $\Vert\cdot\Vert_\mathbb{R}^2=\langle\cdot,\cdot\rangle_\mathbb{R}$, we define the \textit{Clifford module}
\begin{equation*}
V:=\Set{\sum\nolimits_{A\in\mathcal{A}}v_A\otimes e_A | v_A\in V_\mathbb{R}}.
\end{equation*}
For any vector $v=\sum_{A\in\mathcal{A}}v_A\otimes e_A\in V$ and any Clifford number $x=\sum_{A\in\mathcal{A}}x_Ae_A\in\mathbb{R}_n$, we equip this space with a left and a right scalar multiplication
\begin{subequations}
\begin{align}
xv:=&\sum\nolimits_{A,B\in\mathcal{A}}(x_Bv_A)\otimes(e_Be_A), && \textit{(left-multiplication)} \\
vx:=&\sum\nolimits_{A,B\in\mathcal{A}}(v_Ax_B)\otimes(e_Ae_B). && \textit{(right-multiplication)} \label{Eq_Right_multiplication}
\end{align}
\end{subequations}
Moreover, we define the \textit{inner product}
\begin{equation}\label{Eq_Inner_product}
\langle v,w\rangle:=\sum\nolimits_{A,B\in\mathcal{A}}\langle v_A,w_B\rangle_\mathbb{R}\,\overline{e_A}e_B,\qquad v,w\in V,
\end{equation}
and the \textit{norm}
\begin{equation}\label{Eq_Norm}
\Vert v\Vert:=\Big(\sum\nolimits_{A\in\mathcal{A}}\Vert v_A\Vert_\mathbb{R}^2\Big)^{\frac{1}{2}}=\big(\Sc\langle v,v\rangle\big)^{\frac{1}{2}},\qquad v\in V.
\end{equation}
The sesquilinear form \eqref{Eq_Inner_product} is clearly right-linear in the second, and right-antilinear in the first argument, i.e. for every $u,v,w\in V$, $x\in\mathbb{R}_n$, there holds
\begin{align*}
\langle u,v+w\rangle&=\langle u,v\rangle+\langle u,w\rangle, && \langle v,wx\rangle=\langle v,w\rangle x, \\
\langle v+w,u\rangle&=\langle v,u\rangle+\langle w,u\rangle, && \langle vx,w\rangle=\overline{x}\langle v,w\rangle.
\end{align*}
Moreover, there also holds
\begin{equation}\label{Eq_Inner_product_property}
\langle v,xw\rangle=\langle\overline{x}v,w\rangle.
\end{equation}
Next we provide some more basic properties of the inner product \eqref{Eq_Inner_product} and the norm \eqref{Eq_Norm}.

\begin{lem}\label{lem_Properties}
For every $v,w\in V$, $x\in\mathbb{R}_n$, there holds

\begin{enumerate}
\item[i)] $\Vert vx\Vert\leq 2^{\frac{n}{2}}|x|\Vert v\Vert$\qquad and\qquad$\Vert xv\Vert\leq 2^{\frac{n}{2}}|x|\Vert v\Vert$,
\item[ii)] $\Vert vx\Vert=\Vert xv\Vert=|x|\Vert v\Vert$,\qquad if $x\in\mathbb{R}^{n+1}$,
\item[iii)] $|\langle v,w\rangle|\leq 2^{\frac{n}{2}}\Vert v\Vert\,\Vert w\Vert$,
\item[iv)] $|\Sc\langle v,w\rangle|\leq\Vert v\Vert\,\Vert w\Vert$.
\end{enumerate}
\end{lem}

\begin{proof}
\textit{i)}\;\;Analogously to \eqref{Eq_xA}, one can write for every $A\in\mathcal{A}$ the component $(vx)_A$ as
\begin{equation*}
(vx)_A=\Sc(vx\overline{e_A})=\Sc\bigg(\sum\nolimits_{B,C\in\mathcal{A}}(v_B(x\overline{e_A})_C)\otimes(e_Be_C)\bigg)=\sum\nolimits_{B\in\mathcal{A}}(-1)^{\frac{|B|(|B|+1)}{2}}v_B(x\overline{e_A})_B,
\end{equation*}
where in the last equation we used that $e_Be_C\in\mathbb{R}$ if and only if $B=C$. In this form, we can estimate every component of the product as
\begin{align*}
\Vert(vx)_A\Vert_\mathbb{R}&\leq\sum\nolimits_{B\in\mathcal{A}}\Vert v_B(x\overline{e_A})_B\Vert_\mathbb{R}=\sum\nolimits_{B\in\mathcal{A}}\Vert v_B\Vert_\mathbb{R}|(x\overline{e_A})_B| \\
&\leq\Big(\sum\nolimits_{B\in\mathcal{A}}\Vert v_B\Vert^2_\mathbb{R}\Big)^{\frac{1}{2}}\Big(\sum\nolimits_{B\in\mathcal{A}}|(x\overline{e_A})_B|^2\Big)^{\frac{1}{2}}=\Vert v\Vert\,|x\overline{e_A}|=\Vert v\Vert\,|x|.
\end{align*}
Consequently the norm of the whole product admits the upper bound
\begin{equation*}
\Vert vx\Vert^2=\sum\nolimits_{A\in\mathcal{A}}\Vert(vx)_A\Vert_\mathbb{R}^2\leq\sum\nolimits_{A\in\mathcal{A}}\Vert v\Vert^2|x|^2=2^n\Vert v\Vert^2|x|^2.
\end{equation*}
The second inequality $\Vert xv\Vert^2\leq 2^n\Vert v\Vert^2|x|^2$ can be proven in the same way. \medskip

\textit{ii)}\;\;If we assume that $x$ is a paravector, it clearly satisfies $x\overline{x}=|x|^2$, and we have
\begin{equation*}
\Vert vx\Vert^2=\Sc\langle vx,vx\rangle=\Sc\big(\overline{x}\langle v,v\rangle x\big)=\Sc\big(x\overline{x}\langle v,v\rangle\big)=|x|^2\Sc\langle v,v\rangle=|x|^2\Vert v\Vert^2.
\end{equation*}
For the second identity, we use \eqref{Eq_Inner_product_property} and get
\begin{equation*}
\Vert xv\Vert^2=\Sc\langle xv,xv\rangle=\Sc\langle\overline{x}xv,v\rangle=|x|^2\Sc\langle v,v\rangle=|x|^2\Vert v\Vert^2.
\end{equation*}
\textit{iii)}\;\;Using the representation \eqref{Eq_xA} for the Clifford number $\langle v,w\rangle$ as well as the already proven results i) and ii), gives
\begin{equation*}
|\langle v,w\rangle|^2=\sum\nolimits_{A\in\mathcal{A}}(\langle v,w\rangle_A)^2=\sum\nolimits_{A\in\mathcal{A}}\big(\Sc\langle v,w\,\overline{e_A}\rangle\big)^2\leq\sum\nolimits_{A\in\mathcal{A}}\Vert v\Vert^2\Vert w\,\overline{e_A}\Vert^2=2^n\Vert v\Vert^2\Vert w\Vert^2.
\end{equation*}
\textit{iv)}\;\;It follows from \eqref{Eq_SceBeA} and the Cauchy-Schwarz inequality of $\langle\cdot,\cdot\rangle_\mathbb{R}$, that
\begin{align*}
|\Sc\langle u,v\rangle|&=\Big|\Sc\sum\nolimits_{A,B\in\mathcal{A}}\langle v_A,w_B\rangle_\mathbb{R}\overline{e_A}e_B\Big|=\Big|\sum\nolimits_{A\in\mathcal{A}}\langle v_A,w_A\rangle_\mathbb{R}\Big|\leq\sum\nolimits_{A\in\mathcal{A}}\Vert v_A\Vert_\mathbb{R}\Vert w_A\Vert_\mathbb{R} \\
&\leq\Big(\sum\nolimits_{A\in\mathcal{A}}\Vert v_A\Vert_\mathbb{R}^2\Big)^{\frac{1}{2}}\Big(\sum\nolimits_{A\in\mathcal{A}}\Vert w_A\Vert_\mathbb{R}^2\Big)^{\frac{1}{2}}=\Vert v\Vert\,\Vert w\Vert. \qedhere
\end{align*}
\end{proof}

Next, we come to the main result of this first section, the Lax-Milgram lemma for Clifford modules. We will need this result in the later sections to proof unique solvability of the spectral problem for the $S$-spectrum, in particular to detect regions of $S$-resolvent set and related estimates of the $S$-resolvent operator. Although the  Lax-Milgram lemma is a well known result in complex functional analysis, for its Clifford algebra version, we cannot provide a precise reference, hence we give a short proof.

\begin{lem}[Lax-Milgram lemma]\label{lem_Lax_Milgram}
Let $q:V\times V\to\mathbb{R}_n$ be right-linear in the second and right-antilinear in the first argument. Moreover, we assume that there exists constants $C\geq 0$ and $\kappa>0$, such that
\begin{equation*}
\Sc q(v,v)\geq\kappa\Vert v\Vert^2\qquad\text{and}\qquad|q(v,w)|\leq C\Vert v\Vert\Vert w\Vert,\qquad v,w\in V.
\end{equation*}
Then for every right-linear bounded functional $\varphi:V\to\mathbb{R}_n$, there exists a unique vector $u_\varphi\in V$, satisfying
\begin{equation*}
\Vert u_\varphi\Vert\leq\frac{1}{\kappa}\Vert\Sc\varphi\Vert\qquad\text{and}\qquad q(u_\varphi,v)=\varphi(v),\qquad\text{for all }v\in V.
\end{equation*}
\end{lem}

\begin{proof}
Note that it is possible to consider $V$ together with the inner product $\Sc\langle\cdot,\cdot\rangle$ as a real Hilbert space, by simply restricting the scalar multiplication to real numbers. In this sense, the scalar part $\Sc q(\cdot,\cdot):V\times V\to\mathbb{R}$ turns out to be a real bilinear mapping, which satisfies for every $v,w\in V$ the estimates
\begin{equation*}
\Sc q(v,v)\geq\kappa\Vert v\Vert^2\qquad\text{and}\qquad|\Sc q(v,w)|\leq|q(v,w)|\leq C\Vert v\Vert\,\Vert w\Vert.
\end{equation*}
Since also $\Sc\varphi:V\to\mathbb{R}$ is a linear and bounded functional, the classical Lax-Milgram lemma ensures the existence of a unique $u_\varphi\in V$, satisfying
\begin{equation*}
\Vert u_\varphi\Vert\leq\frac{1}{\kappa}\Vert\Sc\varphi\Vert\qquad\text{and}\qquad\Sc q(u_\varphi,v)=\Sc\varphi(v),\qquad\text{for all }v\in V.
\end{equation*}
Using now the components $\varphi(v)_A=\Sc\varphi(v\overline{e_A})$ and $q(u_\varphi,v)_A=\Sc q(u_\varphi,v\overline{e_A})$, see \eqref{Eq_xA}, we can extend this representation result to
\begin{equation*}
q(u_\varphi,v)=\sum\nolimits_{A\in\mathcal{A}}\Sc q(u_\varphi,v\overline{e_A})e_A=\sum\nolimits_{A\in\mathcal{A}}\Sc\varphi(v\overline{e_A})e_A=\varphi(v),\qquad\text{for all }v\in V. \qedhere
\end{equation*}
\end{proof}

We will end this section with two basic lemmas, which will be crucial throughout the paper, namely the Sobolev and the Poincar\'e inequality.
In order to state these inequalities we will need for every function $u\in H^1(\Omega)$ the \textit{Sobolev-seminorm}
\begin{equation*}
\Vert u\Vert_D:=\bigg(\sum_{i=1}^n\Big\Vert\frac{\partial u}{\partial x_i}\Big\Vert_{L^2}^2\bigg)^{1/2}.
\end{equation*}

\begin{lem}[Sobolev inequality]\label{lem_Sobolev_inequality}
Let $n\geq 3$ and $\Omega\subseteq\mathbb{R}^n$ be a minimally smooth domain in the sense of \cite[Section VI.3.3]{S70}. Then there exists a constant $C_S\geq 0$, such that
\begin{subequations}
\begin{align}
\Vert u\Vert_{L^{\frac{2n}{n-2}}}&\leq C_S\Vert u\Vert_D,\hspace{0.95cm}u\in H_0^1(\Omega), \label{Eq_Sobolev_inequality_H01} \\
\Vert u\Vert_{L^{\frac{2n}{n-2}}}&\leq C_S\Vert u\Vert_{H^1},\qquad u\in H^1(\Omega). \label{Eq_Sobolev_inequality_H1}
\end{align}
\end{subequations}
In \eqref{Eq_Sobolev_inequality_H01} the constant is explicitly given by $C_S=\frac{1}{\sqrt{\pi n(n-2)}}\big(\frac{\Gamma(n)}{\Gamma(\frac{n}{2})}\big)^{\frac{1}{n}}$.
\end{lem}

\begin{proof}
The basic Gagliardo-Nirenberg-Sobolev inequality on $\Omega=\mathbb{R}^n$ is
\begin{equation}\label{Eq_Sobolev_inequality_Rn}
\Vert u\Vert_{L^{\frac{2n}{n-2}}}\leq C_S\Vert u\Vert_D,\qquad u\in H^1(\mathbb{R}^n),
\end{equation}
where the optimal constant $C_S$ is for example calculated in \cite{T76}. If we now consider $u\in H_0^1(\Omega)$, the zero-extension of $u$ is an element in $H^1(\mathbb{R}^n)$ and the inequality \eqref{Eq_Sobolev_inequality_H01} follows immediately from \eqref{Eq_Sobolev_inequality_Rn} with the same constant. If on the other hand $u\in H^1(\Omega)$, we have to use the extension operator $E:H^1(\Omega)\to H^1(\mathbb{R}^n)$ from \cite[Section VI.3.1,Theorem~5]{S70}, to get
\begin{equation*}
\Vert u\Vert_{L^{\frac{2n}{n-2}}(\Omega)}\leq\Vert Eu\Vert_{L^{\frac{2n}{n-2}}(\mathbb{R}^n)}\leq C_S\Vert Eu\Vert_{H^1(\mathbb{R}^n)}\leq C_S\Vert E\Vert\Vert u\Vert_{H^1(\Omega)}. \qedhere
\end{equation*}
\end{proof}

\begin{lem}[Poincar\'e inequality]\label{lem_Poincare_inequality}\cite[Corollary 9.19 \& Remark 21]{B11}
Let $\Omega\subseteq\mathbb{R}^n$ be open and either bounded in one direction or of finite measure. Then there exists a constant $C_P>0$, such that
\begin{equation}\label{Eq_Poincare_inequality}
\Vert u\Vert_{L^2}\leq C_P\Vert u\Vert_D,\qquad u\in H_0^1(\Omega).
\end{equation}
\end{lem}

\section{Dirichlet boundary conditions}\label{sec_Dirichlet}

In this section we study the weak formulation of the spectral problem \eqref{Eq_Qs_spectral_problem} with Dirichlet boundary conditions, on some minimally smooth domain $\Omega\subseteq\mathbb{R}^n$, $n\geq 3$, in the sense of Stein \cite[Section VI.3.3]{S70}. Precisely, we equip the form \eqref{Eq_qs_formal} with Dirichlet boundary conditions, that is, for every $s\in\mathbb{R}^{n+1}$ we consider the form
\begin{equation}\label{Eq_qs_Dirichlet}
q_s^\mathsmaller{D}(u,v):=\sum_{i=1}^n\Big\langle\frac{\partial u}{\partial x_i},a_i^2\frac{\partial v}{\partial x_i}+(2s_0a_i-B_i)e_iv\Big\rangle_{L^2}+|s|^2\langle u,v\rangle_{L^2},
\end{equation}
with $\dom q_s^\mathsmaller{D}:=H_0^1(\Omega)\times H_0^1(\Omega)$. We study the connection between regularity and growth conditions of the coefficients of the operator $T$ in \eqref{Eq_T} and the most extended portion of the $S$-resolvent set of $T$.

\begin{prob}
Determine regularity and growth conditions on the coefficients $a_1,\dots,a_n$ in \eqref{Eq_T} and on the spectral parameter $s\in\mathbb{R}^{n+1}$, such that for every $f\in L^2(\Omega)$ there exists a unique solution $u_f\in H_0^1(\Omega)$ of
\begin{equation*}
q_s^\mathsmaller{D}(u_f,v)=\langle f,v\rangle_{L^2},\qquad\text{for all }v\in H_0^1(\Omega).
\end{equation*}
Furthermore, determine $L^2$- and $H^1$-estimates of $u_f$, depending on the parameter $s$.
\end{prob}

\begin{thm}\label{thm_Dirichlet}
Let the bounds of the coefficients $a_1,\dots,a_n$ in \eqref{Eq_Coefficient_bounds} satisfy
\begin{equation*}
m_a^2>C_SM_a',
\end{equation*}
with $C_S\geq 0$ from Lemma~\ref{lem_Sobolev_inequality}. Let us set $K_a:=\frac{M_a}{\sqrt{m_a^2-C_SM_a'}}$ and consider $s\in\mathbb{R}^{n+1}$ with
\begin{equation}\label{Eq_Assumption_Dirichlet}
|s|>K_a|s_0|.
\end{equation}
Then for every $f\in L^2(\Omega)$ there exists a unique $u_f\in H_0^1(\Omega)$ such that
\begin{equation}\label{Eq_Solution_Dirichlet}
q_s^\mathsmaller{D}(u_f,v)=\langle f,v\rangle_{L^2},\qquad\text{for all }v\in H_0^1(\Omega).
\end{equation}
Moreover, this solution satisfies the bounds
\begin{equation}\label{Eq_Estimate_Dirichlet}
\Vert u_f\Vert_{L^2}\leq\frac{\Vert f\Vert_{L^2}}{|s|^2-K_a^2s_0^2}\qquad\text{and}\qquad\Vert u_f\Vert_D\leq\frac{K_a|s|\Vert f\Vert_{L^2}}{M_a(|s|^2-K_a^2s_0^2)}.
\end{equation}
A visualization of these estimates can be seen in Figure 1.
\end{thm}

\begin{figure}
\begin{minipage}{0.49\textwidth}
\begin{center}
\begin{tikzpicture}[scale=2]
\fill[black!70] (-1.508,1)--(1.508,-1)--(-1.508,-1)--(1.508,1);
\draw[fill=black!60] plot[domain=-1.489:1.489] (\x,{(0.025+0.44*(\x)^2)^0.5});
\draw[fill=black!60] plot[domain=-1.489:1.489] (\x,{-(0.025+0.44*(\x)^2)^0.5});
\draw[fill=black!50] plot[domain=-1.469:1.469] (\x,{(0.05+0.44*(\x)^2)^0.5});
\draw[fill=black!50] plot[domain=-1.469:1.469] (\x,{-(0.05+0.44*(\x)^2)^0.5});
\draw[fill=black!40] plot[domain=-1.43:1.43] (\x,{(0.1+0.44*(\x)^2)^0.5});
\draw[fill=black!40] plot[domain=-1.43:1.43] (\x,{-(0.1+0.44*(\x)^2)^0.5});
\draw[fill=black!30] plot[domain=-1.348:1.348] (\x,{(0.2+0.44*(\x)^2)^0.5});
\draw[fill=black!30] plot[domain=-1.348:1.348] (\x,{-(0.2+0.44*(\x)^2)^0.5});
\draw[fill=black!20] plot[domain=-1.168:1.168] (\x,{(0.4+0.44*(\x)^2)^0.5});
\draw[fill=black!20] plot[domain=-1.168:1.168] (\x,{-(0.4+0.44*(\x)^2)^0.5});
\draw[fill=black!10] plot[domain=-0.674:0.674] (\x,{(0.8+0.44*(\x)^2)^0.5});
\draw[fill=black!10] plot[domain=-0.674:0.674] (\x,{-(0.8+0.44*(\x)^2)^0.5});
\draw[ultra thick] (-1.508,1)--(1.508,-1);
\draw[ultra thick] (-1.508,-1)--(1.508,1);
\draw (0.4,0.23) node[anchor=west] {\tiny{$|s|=K_a|s_0|$}};
\draw[ultra thick,white] (-1.55,1.01)--(1.55,1.01);
\draw[ultra thick,white] (-1.55,-1.01)--(1.55,-1.01);
\draw[->] (-1.5,0)--(1.5,0) node[anchor=south] {\tiny{$s_0$}};
\draw[->] (0,0.99)--(0,1.2);
\draw (0,1.1) node[anchor=west] {\tiny{$\Im(s)$}};
\end{tikzpicture}
\end{center}
\end{minipage}
\begin{minipage}{0.49\textwidth}
\begin{center}
\begin{tikzpicture}[scale=2]
\fill[black!70] (-1.508,1)--(1.508,-1)--(-1.508,-1)--(1.508,1);
\draw[fill=black!60] plot[domain=-1.44:1.44] (\x,{((0.025+1.2*(0.000434+(\x)^2)^(0.5))^2-(\x)^2)^0.5});
\draw[fill=black!60] plot[domain=-1.44:1.44] (\x,{-((0.025+1.2*(0.000434+(\x)^2)^(0.5))^2-(\x)^2)^0.5});
\draw[fill=black!50] plot[domain=-1.374:1.374] (\x,{((0.05+1.2*(0.00174+(\x)^2)^(0.5))^2-(\x)^2)^0.5});
\draw[fill=black!50] plot[domain=-1.374:1.374] (\x,{-((0.05+1.2*(0.00174+(\x)^2)^(0.5))^2-(\x)^2)^0.5});
\draw[fill=black!40] plot[domain=-1.244:1.244] (\x,{((0.1+1.2*(0.00694+(\x)^2)^(0.5))^2-(\x)^2)^0.5});
\draw[fill=black!40] plot[domain=-1.244:1.244] (\x,{-((0.1+1.2*(0.00694+(\x)^2)^(0.5))^2-(\x)^2)^0.5});
\draw[fill=black!30] plot[domain=-1:1] (\x,{((0.2+1.2*(0.0278+(\x)^2)^(0.5))^2-(\x)^2)^0.5});
\draw[fill=black!30] plot[domain=-1:1] (\x,{-((0.2+1.2*(0.0278+(\x)^2)^(0.5))^2-(\x)^2)^0.5});
\draw[fill=black!20] plot[domain=-0.495:0.495] (\x,{((0.4+1.2*(0.111+(\x)^2)^(0.5))^2-(\x)^2)^0.5});
\draw[fill=black!20] plot[domain=-0.495:0.495] (\x,{-((0.4+1.2*(0.111+(\x)^2)^(0.5))^2-(\x)^2)^0.5});
\draw[ultra thick] (-1.508,1)--(1.508,-1);
\draw[ultra thick] (-1.508,-1)--(1.508,1);
\draw[ultra thick,white] (-1.55,1.01)--(1.55,1.01);
\draw[ultra thick,white] (-1.55,-1.01)--(1.55,-1.01);
\draw[->] (-1.5,0)--(1.5,0) node[anchor=south] {\tiny{$s_0$}};
\draw[->] (0,0.99)--(0,1.2);
\draw (0,1.1) node[anchor=west] {\tiny{$\Im(s)$}};
\end{tikzpicture}
\end{center}
\end{minipage}
\caption{\small Since the $S$-spectrum is axially symmetric we only visualize its intersection with the plane $\mathbb{C}_J$ for some arbitrary $J\in\mathbb{S}$. In the pictures are plotted the lines $|s|=K_a|s_0|$ obtained by the condition \eqref{Eq_Assumption_Dirichlet}. These lines separate the dark cone region, which is part of the $S$-resolvent set, from the white region that contains the $S$-spectrum (the $S$-spectrum can be strictly contained in the white zone). The two pictures also visualize the level lines of the coefficients $\tau_1(s):=(|s|^2-K_a^2s_0^2)^{-1}$ and $\tau_2(s):=K_a|s|(M_a(|s|^2-K_a^2s_0^2))^{-1}$ from \eqref{Eq_Estimate_Dirichlet}, $\tau_1$ on the left and $\tau_2$ on the right. Here, dark colors indicate large values, light colors indicate small values. In particular, one sees that $\tau_1$ and $\tau_2$ tend to infinity, as $s$ approaches the lines $|s|=K_a|s_0|$. Also the $\mathcal{O}(|s|^{-2})$-behaviour of $\tau_1$ compared to the slower $\mathcal{O}(|s|^{-1})$-behaviour of $\tau_2$ for $|s|\to\infty$ can be seen.}
\end{figure}

\begin{proof}
For the proof of the weak solvability \eqref{Eq_Solution_Dirichlet}, we will verify that the form $q_s^\mathsmaller{D}$ satisfies the assumptions of Lemma~\ref{lem_Lax_Milgram}. For the boundedness, we use the Cauchy-Schwarz type inequality in Lemma~\ref{lem_Properties} iii), as well as the bound $M_a$ in \eqref{Eq_Ma}, to obtain
\begin{align*}
|q_s^\mathsmaller{D}(u,v)|&\leq 2^{\frac{n}{2}}\sum_{i=1}^n\Big\Vert\frac{\partial u}{\partial x_i}\Big\Vert_{L^2}\Big(M_a^2\Big\Vert\frac{\partial v}{\partial x_i}\Big\Vert_{L^2}+2M_a|s_0|\Vert v\Vert_{L^2}+\Vert B_ie_iv\Vert_{L^2}\Big)+2^{\frac{n}{2}}|s|^2\Vert u\Vert_{L^2}\Vert v\Vert_{L^2} \\
&\leq 2^{\frac{n}{2}}\Vert u\Vert_D\bigg(M_a^2\Vert v\Vert_D+2M_a|s_0|\Vert v\Vert_{L^2}+\bigg(\sum_{i=1}^n\Vert B_ie_iv\Vert_{L^2}^2\bigg)^{1/2}\bigg)+2^{\frac{n}{2}}|s|^2\Vert u\Vert_{L^2}\Vert v\Vert_{L^2}.
\end{align*}
Since $B_i$ and $e_i$ are both paravectors, we conclude from Lemma~\ref{lem_Properties} ii), that
\begin{equation*}
|B_ie_iv|=|B_i||e_iv|=|B_i||e_i|\Vert v\Vert=|B_i|\Vert v\Vert.
\end{equation*}
Hölder's inequality with $p=\frac{n}{2}$ and $q=\frac{2}{2-n}$, the Sobolev inequality \eqref{Eq_Sobolev_inequality_H01}, as well as the bound $M_a'$ in \eqref{Eq_Maprime}, then give the estimate of the term
\begin{align}
\sum_{i=1}^n\Vert B_ie_iv\Vert^2_{L^2}&=\sum_{i=1}^n\int_\Omega|B_ie_iv|^2dx=\sum_{i=1}^n\int_\Omega|B_i|^2|v|^2dx=\sum_{i,j=1}^n\int_\Omega\Big|a_j\frac{\partial a_i}{\partial x_j}\Big|^2|v|^2dx \notag \\
&\leq\sum_{i,j=1}^n\Big\Vert a_j\frac{\partial a_i}{\partial x_j}\Big\Vert_{L^n}^2\Vert v\Vert_{L^{\frac{2n}{n-2}}}^2\leq C_S^2\sum_{i,j=1}^n\Big\Vert a_j\frac{\partial a_i}{\partial x_j}\Big\Vert^2_{L^n}\Vert v\Vert^2_D\leq C_S^2M_a'^2\Vert v\Vert^2_D. \label{Eq_Dirichlet_5}
\end{align}
Hence, we conclude the boundedness
\begin{align}
|q_s^\mathsmaller{D}(u,v)|&\leq 2^{\frac{n}{2}}\Vert u\Vert_D\big(M_a^2\Vert v\Vert_D+2M_a|s_0|\Vert v\Vert_{L^2}+C_SM_a'\Vert v\Vert_D\big)+2^{\frac{n}{2}}|s|^2\Vert u\Vert_{L^2}\Vert v\Vert_{L^2} \notag \\
&\leq 2^{\frac{n}{2}}\big(M_a^2+2M_a|s_0|+C_SM_a'+|s|^2\big)\Vert u\Vert_{H^1}\Vert v\Vert_{H^1}. \label{Eq_Dirichlet_11}
\end{align}
Next, for the coercivity of the scalar part $\Sc q_s^\mathsmaller{D}$, we use the lower bound $m_a$ in \eqref{Eq_ma} and the Cauchy-Schwarz type inequality of Lemma~\ref{lem_Properties} iv), to get
\begin{align}
\Sc q_s^\mathsmaller{D}(u,u)&\geq\sum\limits_{i=1}^n\bigg(m_a^2\Big\Vert\frac{\partial u}{\partial x_i}\Big\Vert_{L^2}^2-2M_a|s_0|\Big\Vert\frac{\partial u}{\partial x_i}\Big\Vert_{L^2}\Vert u\Vert_{L^2}-\Big\Vert\frac{\partial u}{\partial x_i}\Big\Vert_{L^2}\Vert B_ie_iu\Vert_{L^2}\bigg)+|s|^2\Vert u\Vert^2_{L^2} \notag \\
&\geq(m_a^2-C_SM_a')\Vert u\Vert_D^2-2M_a|s_0|\Vert u\Vert_D\Vert u\Vert_{L^2}+|s|^2\Vert u\Vert^2_{L^2}, \label{Eq_Dirichlet_1}
\end{align}
where in the second inequality we again used \eqref{Eq_Dirichlet_5}. By the assumption \eqref{Eq_Assumption_Dirichlet} on $s$, there is
\begin{equation}\label{Eq_Dirichlet_8}
\delta_0:=\frac{K_a^2|s_0|}{M_a}<\frac{|s|^2}{M_a|s_0|}=:\delta_1.
\end{equation}
Note, that we interpret $\delta_1=\infty$ if $s_0=0$. Hence, for $\delta\in(\delta_0,\delta_1)$, we use Young's inequality
\begin{equation*}
\Vert u\Vert_D\Vert u\Vert_{L^2}\leq\frac{1}{2\delta}\Vert u\Vert_D^2+\frac{\delta}{2}\Vert u\Vert^2_{L^2},
\end{equation*}
to further estimate the lower bound \eqref{Eq_Dirichlet_1} by
\begin{equation}\label{Eq_Dirichlet_2}
\Sc q_s^\mathsmaller{D}(u,u)\geq\Big(\frac{M_a^2}{K_a^2}-\frac{M_a|s_0|}{\delta}\Big)\Vert u\Vert^2_D+\big(|s|^2-M_a|s_0|\delta\big)\Vert u\Vert^2_{L^2}
\end{equation}
Note, that $\delta_0$ and $\delta_1$ in \eqref{Eq_Dirichlet_8} are chosen exactly in the way that both brackets on the right hand side of \eqref{Eq_Dirichlet_2} are positive. Hence we have proven that $\Sc q_s^\mathsmaller{D}$ is coercive in $H_0^1(\Omega)$. Fixing now any function $f\in L^2(\Omega)$, we can consider the corresponding functional
\begin{equation*}
\varphi_f(v):=\langle f,v\rangle_{L^2},\qquad v\in H_0^1(\Omega).
\end{equation*}
Then by Lemma~\ref{lem_Properties}~iii), this functional is bounded in $H_0^1(\Omega)$, i.e.
\begin{equation}\label{Eq_Dirichlet_3}
|\varphi_f(v)|=|\langle f,v\rangle_{L^2}|\leq 2^{\frac{n}{2}}\Vert f\Vert_{L^2}\Vert v\Vert_{L^2}\leq 2^{\frac{n}{2}}\Vert f\Vert_{L^2}\Vert v\Vert_{H^1},\qquad v\in H_0^1(\Omega).
\end{equation}
Hence, the assumptions of Lemma~\ref{lem_Lax_Milgram} are satisfied and there exists a unique weak solution $u_f\in H_0^1(\Omega)$, which satisfies
\begin{equation}\label{Eq_Dirichlet_4}
q_s^\mathsmaller{D}(u_f,v)=\varphi_f(v)=\langle f,v\rangle_{L^2},\qquad\text{for all }v\in H_0^1(\Omega).
\end{equation}
For the $L^2$-estimate in \eqref{Eq_Estimate_Dirichlet}, we test \eqref{Eq_Dirichlet_4} with $v=u_f$. Using also Lemma~\ref{lem_Properties} iv), this gives
\begin{equation}\label{Eq_Dirichlet_7}
\Sc q_s^\mathsmaller{D}(u_f,u_f)=\Sc\langle f,u_f\rangle_{L^2}\leq\Vert f\Vert_{L^2}\Vert u_f\Vert_{L^2}.
\end{equation}
Combining this inequality with the coercivity estimate \eqref{Eq_Dirichlet_2}, furthermore gives
\begin{equation*}
\big(|s|^2-M_a|s_0|\delta\big)\Vert u_f\Vert^2_{L^2}\leq\Sc q_s^\mathsmaller{D}(u_f,u_f)\leq\Vert f\Vert_{L^2}\Vert u_f\Vert_{L^2}.
\end{equation*}
However, since $\delta\in(\delta_0,\delta_1)$ was chosen arbitrary, we can maximize the left hand side by taking the limit $\delta\to\delta_0=\frac{K_a^2|s_0|}{M_a}$, and get
\begin{equation}\label{Eq_Dirichlet_9}
\Vert u_f\Vert_{L^2}\leq\frac{\Vert f\Vert_{L^2}}{|s|^2-K_a^2s_0^2}.
\end{equation}
For the $\Vert\cdot\Vert_D$-estimate, we get from \eqref{Eq_Dirichlet_2}, \eqref{Eq_Dirichlet_7} and \eqref{Eq_Dirichlet_9} for every $\delta\in(\delta_0,\delta_1)$ the estimate
\begin{equation*}
\Big(\frac{M_a^2}{K_a^2}-\frac{M_a|s_0|}{\delta}\Big)\Vert u_f\Vert^2_D\leq\Sc q_s^D(u_f,u_f)\leq\Vert f\Vert_{L^2}\Vert u_f\Vert_{L^2}\leq\frac{\Vert f\Vert_{L^2}^2}{|s|^2-K_a^2s_0^2}.
\end{equation*}
Maximizing the left hand side with respect to $\delta\in(\delta_0,\delta_1)$, i.e. taking the limit $\delta\to\delta_1=\frac{|s|^2}{M_a|s_0|}$, gives
\begin{equation*}
\Big(\frac{M_a^2}{K_a^2}-\frac{M_a^2s_0^2}{|s|^2}\Big)\Vert u_f\Vert^2_D\leq\Vert f\Vert_{L^2}\Vert u_f\Vert_{L^2}\leq\frac{\Vert f\Vert_{L^2}^2}{|s|^2-K_a^2s_0^2}
\end{equation*}
Rearranging this inequality gives exactly the second estimate in \eqref{Eq_Estimate_Dirichlet}.
\end{proof}

In the special case of minimally smooth domains $\Omega\subseteq\mathbb{R}^n$, which are either bounded in one direction or of finite measure, we even get stronger results. In particular, the double sector of Theorem~\ref{thm_Dirichlet} is modified and disconnected at the origin $s=0$, which means that the Dirichlet form $q_s^\mathsmaller{D}$ is uniquely solvable also in a vertical strip around the imaginary axis.

\begin{thm}\label{thm_Dirichlet_bounded}
Let $\Omega\subseteq\mathbb{R}^n$ be a bounded minimally smooth domain in $n\geq 3$ dimensions and let the bounds \eqref{Eq_Coefficient_bounds} of the coefficients $a_1,\dots,a_n$ in \eqref{Eq_Coefficient_bounds} satisfy
\begin{equation*}
m_a^2>C_SM_a',
\end{equation*}
with $C_S\geq 0$ from Lemma~\ref{lem_Sobolev_inequality}. Let us set $K_a:=\frac{M_a}{\sqrt{m_a^2-C_SM_a'}}$, use the Poincar\'e constant $C_P\in(0,1]$ from Lemma~\ref{lem_Poincare_inequality}, and consider $s\in\mathbb{R}^{n+1}$ with
\begin{equation}\label{Eq_Assumption_Dirichlet_bounded}
|s|^2>\begin{cases} \frac{2M_a|s_0|}{C_P}-\frac{M_a^2}{C_P^2K_a^2}, & \text{if }|s_0|\leq\frac{M_a}{C_PK_a^2}, \\ K_a^2s_0^2, & \text{if }|s_0|\geq\frac{M_a}{C_PK_a^2}. \end{cases}
\end{equation}
Then for every $f\in L^2(\Omega)$ there exists a unique $u_f\in H_0^1(\Omega)$, such that
\begin{equation}\label{Eq_Solution_Dirichlet_bounded}
q_s^\mathsmaller{D}(u_f,v)=\langle f,v\rangle_{L^2},\qquad {\rm for\ all}\ \  v\in H_0^1(\Omega).
\end{equation}
Moreover, this solution admits the norm estimates
\begin{equation}\label{Eq_Estimate_Dirichlet_bounded}
\Vert u_f\Vert_{L^2}\leq\frac{\Vert f\Vert_{L^2}}{\kappa_0(s)}\qquad\text{and}\qquad\Vert u_f\Vert_D\leq\frac{\Vert f\Vert_{L^2}}{\sqrt{\kappa_0(s)\kappa_1(s)}},
\end{equation}
where
\begin{align}
\kappa_0(s):=&\begin{cases} |s|^2-\frac{2M_a|s_0|}{C_P}+\frac{M_a^2}{C_P^2K_a^2}, & \text{if }|s_0|\leq\frac{M_a}{C_PK_a^2}, \\ |s|^2-K_a^2s_0^2, & \text{if }|s_0|\geq\frac{M_a}{C_PK_a^2}, \end{cases} \label{Eq_kappa0} \\
\kappa_1(s):=&\begin{cases} C_P^2\big(|s|^2-\frac{2M_a|s_0|}{C_P}+\frac{M_a^2}{C_P^2K_a^2}\big), & \text{if }|s|^2\leq\frac{M_a|s_0|}{C_P}, \\ \frac{M_a^2(|s|^2-K_a^2s_0^2)}{K_a^2|s|^2}, & \text{if }|s|^2\geq\frac{M_a|s_0|}{C_P}. \end{cases} \label{Eq_kappa1}
\end{align}
\end{thm}

\vspace{-0.4cm}

\begin{figure}[H]
\begin{minipage}{0.49\textwidth}
\begin{center}
\begin{tikzpicture}[scale=2]
\draw[->] (-1.5,0)--(1.5,0) node[anchor=south] {\tiny{$s_0$}};
\draw[ultra thick,fill=black!70] (1.508,-1)--(0.556,-0.369) arc (236.44:123.56:0.442)--(1.508,1)--(-1.508,1)--(-0.556,0.369) arc (56.44:-56.44:0.442)--(-1.508,-1);
\draw[fill=black!60] (1.477,1)--(0.556,0.419) arc (120.24:239.76:0.485)--(1.477,-1)--(-1.477,-1)--(-0.556,-0.419) arc (-59.76:59.76:0.485)--(-1.477,1);
\draw[fill=black!50] (1.446,1)--(0.556,0.465) arc (117.75:242.25:0.525)--(1.446,-1)--(-1.446,-1)--(-0.556,-0.465) arc (-62.25:62.25:0.525)--(-1.446,1);
\draw[fill=black!40] (1.382,1)--(0.556,0.544) arc (114.2:245.8:0.596)--(1.382,-1)--(-1.382,-1)--(-0.556,-0.544) arc (-65.80:65.80:0.596)--(-1.382,1);
\draw[fill=black!30] (1.243,1)--(0.556,0.675) arc (109.9:250.1:0.718)--(1.243,-1)--(-1.243,-1)--(-0.556,-0.675) arc (-70.1:70.1:0.718)--(-1.243,1);
\draw[fill=black!20] (0.904,1)--(0.556,0.881) arc (105.51:151.58:0.914) arc (28.92:74.49:0.914)--(-0.904,1);
\draw[fill=black!20] (0.904,-1)--(0.556,-0.881) arc (254.49:208.92:0.914) arc (-28.92:-74.49:0.914)--(-0.904,-1);
\draw[fill=black!10] (0.172,1) arc (122.14:132.64:1.181) arc (47.36:57.86:1.181);
\draw[fill=black!10] (0.172,-1) arc (237.86:227.36:1.181) arc (-47.36:-57.86:1.181);
\draw[dashed] (0.556,0.369)--(0.556,1)--(-0.556,1)--(-0.556,0.369);
\draw[dashed] (0.556,-0.369)--(0.556,-1)--(-0.556,-1)--(-0.556,-0.369);
\draw[ultra thick,white] (-1.55,1.01)--(1.55,1.01);
\draw[ultra thick,white] (-1.55,-1.01)--(1.55,-1.01);
\draw (0.8,0.05)--(0.8,-0.05);
\draw (0.8,0) node[anchor=north] {\tiny{$\frac{M_a}{C_P}$}};
\draw[->] (0.8,0)--(0.43,0.22);
\draw (0.44,0.25) node[anchor=west] {\tiny{$\frac{M_a\sqrt{1-K_a^{-2}}}{C_P}$}};
\draw (0.8,0.5) node[anchor=west] {\tiny{$|s|=K_a|s_0|$}};
\draw[->] (0,0.99)--(0,1.2);
\draw (0,1.1) node[anchor=west] {\tiny{$\Im(s)$}};
\end{tikzpicture}
\end{center}
\end{minipage}
\begin{minipage}{0.49\textwidth}
\begin{center}
\begin{tikzpicture}[scale=2]
\draw[->] (-1.5,0)--(1.5,0) node[anchor=south] {\tiny{$s_0$}};
\draw[ultra thick,fill=black!70] (1.508,-1)--(0.556,-0.369) arc (236.44:123.56:0.442)--(1.508,1)--(-1.508,1)--(-0.556,0.369) arc (56.44:-56.44:0.442)--(-1.508,-1);
\draw[fill=black!60] (1.4,1)--(0.556,0.439)--(0.489,0.39) arc (128.6:231.4:0.499)--(0.556,-0.439)--(1.4,-1)--(-1.4,-1)--(-0.556,-0.439)--(-0.489,-0.39) arc (-51.4:51.4:0.499)--(-0.489,0.39)--(-1.4,1);
\draw[fill=black!50] (1.295,1)--(0.556,0.506) arc (122.86:134.34:0.857) arc (135.72:224.28:0.57) arc (225.66:237.14:0.857)--(1.295,-1)--(-1.295,-1)--(-0.556,-0.506) arc (-57.14:-45.66:0.857) arc (-44.28:44.28:0.57) arc (45.66:57.14:0.857)--(-1.295,1);
\draw[fill=black!40] (1.093,1)--(0.556,0.637) arc (123.29:143.62:1.042) arc (143.06:216.94:0.639) arc (216.38:236.71:1.042)--(1.093,-1)--(-1.093,-1)--(-0.556,-0.637) arc (-56.71:-36.38:1.042) arc (-36.94:36.94:0.639) arc (36.38:56.71:1.042)--(-1.093,1);
\draw[fill=black!30] (0.703,1)--(0.556,0.902) arc (120.65:169.96:1.124) arc (167.23:192.77:0.593) arc (190.04:239.35:1.124)--(0.703,-1)--(-0.703,-1)--(-0.556,-0.902) arc (-59.35:-10.04:1.124) arc (-12.77:12.77:0.593) arc (10.04:59.35:1.124)--(-0.703,1);
\draw[fill=black!20] (0.21,1)--(0,0.798)--(-0.21,1);
\draw[fill=black!20] (0.21,-1)--(0,-0.798)--(-0.21,-1);
\draw[dashed] (0.556,1)--(0.556,0.369) arc (67.11:292.89:0.4)--(0.556,-1)--(-0.556,-1)--(-0.556,-0.369) arc (-112.89:112.89:0.4)--(-0.556,1);
\draw (-0.556,0.05)--(-0.556,-0.05);
\draw (-0.65,0) node[anchor=north] {\tiny{$\frac{M_a}{C_PK_a^2}$}};
\draw (0.42,0.05)--(0.42,-0.05);
\draw (0.54,0.03) node[anchor=north] {\tiny{$\frac{M_a}{2C_P}$}};
\draw[->] (0.42,0)--(0.117,0.283) node[anchor=south] {\tiny{$\frac{M_a}{2C_P}$}};
\draw[ultra thick,white] (-1.55,1.01)--(1.55,1.01);
\draw[ultra thick,white] (-1.55,-1)--(1.55,-1);
\draw[->] (0,0.99)--(0,1.2);
\draw (0,1.1) node[anchor=west] {\tiny{$\Im(s)$}};
\end{tikzpicture}
\end{center}
\end{minipage}
\caption{\small Since the $S$-spectrum is axially symmetric we only visualize its intersection with the plane $\mathbb{C}_J$ for some arbitrary $J\in\mathbb{S}$. The dark color in the two graphics indicates a region of points in the $S$-resolvent set, while the white region contains the $S$-spectrum (but can be strictly contained in it). The two pictures also visualize the level lines of the coefficients $\tau_1(s):=\kappa_0(s)^{-1}$ and $\tau_2(s):=(\kappa_0(s)\kappa_1(s))^{-1/2}$ appearing in the estimate  \eqref{Eq_Estimate_Dirichlet_bounded}, $\tau_1$ on the left and $\tau_2$ on the right. Here, dark colors indicate large values, light colors indicate small values. In particular, one sees that $\tau_1$ and $\tau_2$ tend to infinity as $s$ approaches the boundary of the dark region. Also the $\mathcal{O}(|s|^{-2})$-behaviour of $\tau_1$ compared to the slower $\mathcal{O}(|s|^{-1})$-behaviour of $\tau_2$ for $|s|\to\infty$ can be seen. Compared with Figure 1, we see that here the $S$-spectrum has a gap at the origin which is generated by the use of the Poincar\'e inequality.}
\end{figure}

\begin{proof}
In the same way as in \eqref{Eq_Dirichlet_11}, one proves the boundedness
\begin{equation*}
|q_s^\mathsmaller{D}(u,v)|\leq 2^{\frac{n}{2}}\big(M_a^2+2M_a|s_0|+C_SM_a'+|s|^2\big)\Vert u\Vert_{H^1}\Vert v\Vert_{H^1},\qquad u,v\in H_0^1(\Omega).
\end{equation*}
Also for the coercivity of $\Sc q_s^\mathsmaller{D}$, we start as in \eqref{Eq_Dirichlet_1} with the inequality
\begin{equation}\label{Eq_Dirichlet_bounded_1}
\Sc q_s(u,u)\geq\frac{M_a^2}{K_a^2}\Vert u\Vert^2_D-2M_a|s_0|\Vert u\Vert_D\Vert u\Vert_{L^2}+|s|^2\Vert u\Vert^2_{L^2},
\end{equation}
where $K_a^2=\frac{M_a^2}{m_a^2-C_SM_a'}$. For any $\alpha\in[0,1]$ and $\delta>0$, we now use Young's inequality
\begin{equation*}
\Vert u\Vert_D\Vert v\Vert_{L^2}\leq\frac{C_P}{2\delta}\Vert u\Vert_D^2+\frac{\delta}{2C_P}\Vert u\Vert_{L^2}^2
\end{equation*}
as well as the Poincar\'e inequality \eqref{Eq_Poincare_inequality} in the coercivity estimate \eqref{Eq_Dirichlet_bounded_1}, to get
\begin{align}
\Sc q_s^\mathsmaller{D}(u,u)&\geq\frac{M_a^2}{K_a^2}\Vert u\Vert_D^2-2\alpha M_a|s_0|\Vert u\Vert_D\Vert u\Vert_{L^2}-2(1-\alpha)M_a|s_0|\Vert u\Vert_D\Vert u\Vert_{L^2}+|s|^2\Vert u\Vert^2_{L^2} \notag \\
&\geq M_a\Big(\frac{M_a}{K_a^2}-\frac{\alpha C_P|s_0|}{\delta}-2(1-\alpha)C_P|s_0|\Big)\Vert u\Vert^2_D+\Big(|s|^2-\frac{\alpha\delta M_a|s_0|}{C_P}\Big)\Vert u\Vert^2_{L^2}. \label{Eq_Dirichlet_bounded_4}
\end{align}
In order to get an $H^1$-coercivity lower bound, both brackets on the right hand side have to be positive, i.e. we have to choose $\alpha\in[0,1]$ and $\delta>0$ such that
\begin{align*}
\text{I}:\;\;&\frac{M_a}{K_a^2}-\frac{\alpha C_P|s_0|}{\delta}-2(1-\alpha)C_P|s_0|>0 && \Rightarrow\quad\frac{1}{\delta}<\frac{M_a-2(1-\alpha)C_PK_a^2|s_0|}{\alpha C_PK_a^2|s_0|}, \\
\text{II}:\;\;&|s|^2-\frac{\alpha\delta M_a|s_0|}{C_P}>0 && \Rightarrow\quad\delta<\frac{C_P|s|^2}{\alpha M_a|s_0|}.
\end{align*}
However, such a choice is only possible if
\begin{equation*}
M_a-2(1-\alpha)C_PK_a^2|s_0|>0\qquad\text{and}\qquad\frac{\alpha C_PK_a^2|s_0|}{M_a-2(1-\alpha)C_PK_a^2|s_0|}<\frac{C_P|s|^2}{\alpha M_a|s_0|}.
\end{equation*}
If we now set
\begin{equation}\label{Eq_Dirichlet_bounded_6}
x:=\frac{M_a}{C_PK_a^2|s_0|}\qquad\text{and}\qquad y:=\frac{C_P|s|^2}{M_a|s_0|},
\end{equation}
these conditions turn into
\begin{equation}\label{Eq_Dirichlet_bounded_7}
x-2+2\alpha>0\qquad\text{and}\qquad(x-2)y+2\alpha y-\alpha^2>0.
\end{equation}
Taking now Lemma \ref{lem_Maximization_problem}, the second condition is satisfied if and only if
\begin{equation*}
y>\begin{cases} \frac{1}{x}, & \text{if }x\leq 1, \\ 2-x, & \text{if }x\geq 1, \end{cases}
\end{equation*}
which, plugging back in \eqref{Eq_Dirichlet_bounded_6}, turns into the assumed condition \eqref{Eq_Assumption_Dirichlet_bounded}. Moreover, if we choose the values $\alpha_0<\alpha_1\in[0,1]$ from \eqref{Eq_alpha0_alpha1}, then for every $\alpha\in(\alpha_0,\alpha_1)$ also the first condition in \eqref{Eq_Dirichlet_bounded_7} is satisfied. Altogether, for every $\alpha\in(\alpha_0,\alpha_1)$ and $\delta$ in between
\begin{equation*}
\delta_0(\alpha):=\frac{\alpha C_PK_a^2|s_0|}{M_a-2(1-\alpha)C_PK_a^2|s_0|}<\delta<\frac{C_P|s|^2}{\alpha M_a|s_0|}=:\delta_1(\alpha),
\end{equation*}
both brackets in \eqref{Eq_Dirichlet_bounded_4} are positive and we have proven the coercivity of the scalar part $\Sc q_s^\mathsmaller{D}$. Let now $f\in L^2(\Omega)$, and consider the functional $\varphi_f(v):=\langle f,v\rangle_{L^2}$, $v\in H_0^1(\Omega)$, which is bounded in $H_0^1(\Omega)$ by the same estimate as in \eqref{Eq_Dirichlet_3}. Lemma~\ref{lem_Lax_Milgram} then states the existence of a unique $u_f\in H_0^1(\Omega)$, which satisfies
\begin{equation}\label{Eq_Dirichlet_bounded_5}
q_s^\mathsmaller{D}(u_f,v)=\varphi_f(v)=\langle f,v\rangle_{L^2},\qquad\text{for all }v\in H_0^1(\Omega).
\end{equation}
Next, we will prove the bounds \eqref{Eq_Estimate_Dirichlet_bounded} of this solution. Using the Poincar\'e inequality \eqref{Eq_Poincare_inequality} in the form $\Vert u_f\Vert_D\geq\frac{1}{C_P}\Vert u_f\Vert_{L^2}$ on the coercivity estimate \eqref{Eq_Dirichlet_bounded_4}, gives
\begin{equation}\label{Eq_Dirichlet_bounded_8}
\Sc q_s^\mathsmaller{D}(u_f,u_f)\geq\Big(|s|^2-\frac{2M_a|s_0|}{C_P}+\frac{M_a^2}{C_P^2K_a^2}+\frac{\alpha M_a|s_0|}{C_P}\Big(2-\frac{1}{\delta}-\delta\Big)\Big)\Vert u_f\Vert^2_{L^2}.
\end{equation}
This lower bound can now be maximized with respect to the two free parameters $\alpha\in(\alpha_0,\alpha_1)$ and $\delta\in(\delta_0(\alpha),\delta_1(\alpha))$. However, the maximization problem \eqref{Eq_Maximization_problem2} with the values \eqref{Eq_Dirichlet_bounded_6}, already states that
\begin{equation*}
\sup\limits_{\stackrel{\alpha\in(\alpha_0,\alpha_1)}{\mathsmaller{\delta\in(\delta_0(\alpha),\delta_1(\alpha))}}}\alpha\Big(2-\frac{1}{\delta}-\delta\Big)=\begin{cases} 0, & \text{if }|s_0|\leq\frac{M_a}{C_PK_a^2}, \\ 2-\frac{M_a}{C_PK_a^2|s_0|}-\frac{C_PK_a^2|s_0|}{M_a}, & \text{if }|s_0|\geq\frac{M_a}{C_PK_a^2}. \end{cases}
\end{equation*}
This maximization problem now transforms \eqref{Eq_Dirichlet_bounded_8} into
\begin{equation*}
\Sc q_s^\mathsmaller{D}(u_f,u_f)\geq\Vert u_f\Vert_{L^2}^2\begin{cases} |s|^2-\frac{2M_a|s_0|}{C_P}+\frac{M_a^2}{C_P^2K_a^2}, & \text{if }|s_0|\leq\frac{M_a}{C_PK_a^2}, \\ |s|^2-K_a^2s_0^2, & \text{if }|s_0|\geq\frac{M_a}{C_PK_a^2}, \end{cases}=\kappa_0(s)\Vert u_f\Vert_{L^2}^2,
\end{equation*}
where in the second equality we used the constant $\kappa_0(s)$ from \eqref{Eq_kappa0}. Combining this inequality with \eqref{Eq_Dirichlet_7}, then gives the final $L^2$-estimate of the solution
\begin{equation}\label{Eq_Dirichlet_bounded_2}
\Vert u_f\Vert_{L^2}^2\leq\frac{\Sc q_s^\mathsmaller{D}(u_f,u_f)}{\kappa_0(s)}\leq\frac{\Vert f\Vert_{L^2}\Vert u_f\Vert_{L^2}}{\kappa_0(s)},
\end{equation}
which is exactly the $\Vert\cdot\Vert_{L^2}$-norm estimate in \eqref{Eq_Estimate_Dirichlet_bounded}. For the $\Vert\cdot\Vert_D$-estimate of the solution $u_f$, we once more use the coercivity \eqref{Eq_Dirichlet_bounded_4}, use the nonnegativity of the second term, and get
\begin{equation*}
\Sc q_s^\mathsmaller{D}(u_f,u_f)\geq\Big(\frac{M_a^2}{K_a^2}-\frac{\alpha C_PM_a|s_0|}{\delta}-2(1-\alpha)C_PM_a|s_0|\Big)\Vert u_f\Vert_D^2.
\end{equation*}
Also this lower bound will now be maximized with respect to the free parameters $\alpha$ and $\delta$. Maximizing with respect to $\delta$, means choosing the maximal possible value, i.e. the limit $\delta\to\delta_1(\alpha)=\frac{C_P|s|^2}{\alpha M_a|s_0|}$. This gives
\begin{equation}\label{Eq_Dirichlet_bounded_3}
\Sc q_s^\mathsmaller{D}(u_f,u_f)\geq\Big(\frac{M_a^2}{K_a^2}-2C_PM_a|s_0|+\alpha C_PM_a|s_0|\Big(2-\frac{\alpha M_a|s_0|}{C_P|s|^2}\Big)\Big)\Vert u_f\Vert_D^2.
\end{equation}
The maximization problem \eqref{Eq_Maximization_problem1} with the values \eqref{Eq_Dirichlet_bounded_6} gives
\begin{equation*}
\sup\limits_{\alpha\in(\alpha_0,\alpha_1)}\alpha\Big(2-\frac{\alpha M_a|s_0|}{C_P|s|^2}\Big)=\begin{cases} \frac{C_P|s|^2}{M_a|s_0|}, & \text{if }|s|^2\leq\frac{M_a|s_0|}{C_P}, \\ 2-\frac{M_a|s_0|}{C_P|s|^2}, & \text{if }|s|^2\geq\frac{M_a|s_0|}{C_P}. \end{cases}
\end{equation*}
This maximization problem allows us to also maximize \eqref{Eq_Dirichlet_3} with respect to $\alpha$, and get
\begin{equation*}
\Sc q_s^\mathsmaller{D}(u_f,u_f)\geq\Vert u_f\Vert_D^2\begin{cases} \frac{M_a^2}{K_a^2}-2C_PM_a|s_0|+C_P^2|s|^2, & \text{if }|s|^2\leq\frac{M_a|s_0|}{C_P}, \\ \frac{M_a^2}{K_a^2}-\frac{M_a^2s_0^2}{|s|^2}, & \text{if }|s|^2\geq\frac{M_a|s_0|}{C_P}, \end{cases}=\kappa_1(s)\Vert u_f\Vert_D^2,
\end{equation*}
where in the second equality we used the constant $\kappa_1(s)$ from \eqref{Eq_kappa1}. Rearranging now this inequality and using again \eqref{Eq_Dirichlet_7} and the already calculated $L^2$-estimate \eqref{Eq_Dirichlet_bounded_2}, finally gives the $\Vert\cdot\Vert_D$-estimate of the solution
\begin{equation*}
\Vert u_f\Vert_D^2\leq\frac{\Sc q_s^\mathsmaller{D}(u_f,u_f)}{\kappa_1(s)}\leq\frac{\Vert f\Vert_{L^2}\Vert u_f\Vert_{L^2}}{\kappa_1(s)}\leq\frac{\Vert f\Vert_{L^2}^2}{\kappa_0(s)\kappa_1(s)}. \qedhere
\end{equation*}
\end{proof}

\section{Robin-like boundary conditions}\label{sec_Robin}

In this section we study the weak formulation of the spectral problem \eqref{Eq_Qs_spectral_problem} with Robin-like boundary conditions
\begin{equation}\label{Eq_Robin_boundary_conditions}
\sum_{i=1}^n\nu_ia_i^2\frac{\partial u}{\partial x_i}+bu=0,\qquad\text{on }\partial\Omega.
\end{equation}
The domain $\Omega\subseteq\mathbb{R}^n$, in $n\geq 3$ dimensions is assumed to be minimally smooth in the sense of Stein \cite[Section VI.3.3]{S70}. Precisely, if we plug in the boundary conditions \eqref{Eq_Robin_boundary_conditions} into \eqref{Eq_qs_formal}, means that for every $s\in\mathbb{R}^{n+1}$ we are dealing with form
\begin{equation}\label{Eq_qs_Robin}
q_s^\mathsmaller{R}(u,v):=\sum\limits_{i=1}^n\Big\langle\frac{\partial u}{\partial x_i},a_i^2\frac{\partial v}{\partial x_i}+(2s_0a_i-B_i)e_iv\Big\rangle_{L^2}+|s|^2\langle u,v\rangle_{L^2}+\langle b\,\tau_\mathsmaller{D}u,\tau_\mathsmaller{D}v\rangle_{L^2(\partial\Omega)},
\end{equation}
defined on $\dom q_s^\mathsmaller{R}:=H^1(\Omega)\times H^1(\Omega)$. Here $\tau_\mathsmaller{D}:H^1(\Omega)\rightarrow L^2(\partial\Omega)$ is the bounded Dirichlet trace operator, see \cite[Equation (4)]{M87}.

\begin{prob}
Determine regularity and growth conditions on the coefficients $a_1,\dots,a_n$ in \eqref{Eq_T}, the Robin parameter $b$ in \eqref{Eq_Robin_boundary_conditions} and on the spectral parameter $s\in\mathbb{R}^{n+1}$, such that for every $f\in L^2(\Omega)$ there exists a unique solution $u_f\in H^1(\Omega)$ of the problem
\begin{equation*}
q_s^\mathsmaller{R}(u_f,v)=\langle f,v\rangle_{L^2},\qquad\text{for all }v\in H^1(\Omega).
\end{equation*}
Furthermore, determine the $L^2$- and $H^1$-estimates of $u_f$, depending on the parameter $s$.
\end{prob}

\begin{thm}\label{thm_Robin}
Let $b\in L^\infty(\partial\Omega)$ be real valued and $a_1,\dots,a_n\in C^1(\overline{\Omega})$ be positive coefficients. With the Sobolev constant $C_S\geq 0$ from \eqref{Eq_Sobolev_inequality_H1} and the upper bound $M_a'$ from \eqref{Eq_Maprime}, we assume that there holds
\begin{equation*}
m_a^2>C_SM_a'+\Vert\hspace{-0.05cm}\min\{b,0\}\Vert_{L^\infty(\partial\Omega)}\Vert\tau_\mathsmaller{D}\Vert^2=:D_{a,b}.
\end{equation*}
We now consider $s\in\mathbb{R}^{n+1}$ with
\begin{equation}\label{Eq_Assumption_Robin}
|s|^2>\frac{M_a^2s_0^2}{m_a^2-D_{a,b}}+D_{a,b}.
\end{equation}
Then for every $f\in L^2(\Omega)$ there exists a unique $u_f\in H^1(\Omega)$ such that
\begin{equation}\label{Eq_Solution_Robin}
q_s^\mathsmaller{R}(u_f,v)=\langle f,v\rangle_{L^2},\qquad\text{for all }v\in H^1(\Omega).
\end{equation}
Moreover, this solution satisfies the bound
\begin{equation}\label{Eq_Estimate_Robin}
\Vert u_f\Vert_{H^1}\leq\frac{2\Vert f\Vert_{L^2}}{|s|^2+m_a^2-2D_{a,b}-\sqrt{(|s|^2-m_a^2)^2+4M_a^2s_0^2}}.
\end{equation}
\end{thm}

\begin{figure}[H]
\begin{center}
\begin{tikzpicture}[scale=3]
\draw[ultra thick,fill=black!70] plot[domain=-1.5:1.5] (\x,{(0.23*(\x)^2+0.03)^(0.5)});
\draw[ultra thick,fill=black!70] plot[domain=-1.5:1.5] (\x,{-(0.23*(\x)^2+0.03)^(0.5)});
\draw[fill=black!60] plot[domain=-1.174:1.174] (\x,{(0.368*(\x)^2+0.043)^(0.5)});
\draw[fill=black!60] plot[domain=-1.174:1.174] (\x,{-(0.368*(\x)^2+0.043)^(0.5)});
\draw[fill=black!50] plot[domain=-0.958:0.958] (\x,{(0.538*(\x)^2+0.056)^(0.5)});
\draw[fill=black!50] plot[domain=-0.958:0.958] (\x,{-(0.538*(\x)^2+0.056)^(0.5)});
\draw[fill=black!40] plot[domain=-0.667:0.667] (\x,{(1.051*(\x)^2+0.082)^(0.5)});
\draw[fill=black!40] plot[domain=-0.667:0.667] (\x,{-(1.051*(\x)^2+0.082)^(0.5)});
\draw[fill=black!30] plot[domain=-0.461:0.461] (\x,{(2.077*(\x)^2+0.108)^(0.5)});
\draw[fill=black!30] plot[domain=-0.461:0.461] (\x,{-(2.077*(\x)^2+0.108)^(0.5)});
\draw[fill=black!20] plot[domain=-0.284:0.284] (\x,{(5.154*(\x)^2+0.134)^(0.5)});
\draw[fill=black!20] plot[domain=-0.284:0.284] (\x,{-(5.154*(\x)^2+0.134)^(0.5)});
\draw[fill=black!10] plot[domain=-0.167:0.167] (\x,{(14.384*(\x)^2+0.15)^(0.5)});
\draw[fill=black!10] plot[domain=-0.167:0.167] (\x,{-(14.384*(\x)^2+0.15)^(0.5)});
\draw[ultra thick,white] (-1.55,0.741)--(1.55,0.741);
\draw[ultra thick,white] (-1.55,-0.741)--(1.55,-0.741);
\draw[->] (-1.5,0)--(1.5,0) node[anchor=south] {\tiny{$s_0$}};
\draw[->] (0,-0.8)--(0,0.85) node[anchor=west] {\tiny{$\Im(s)$}};
\draw (0,0.11) node[anchor=west] {\tiny{$\sqrt{D_{a,b}}$}};
\draw (1.3,0.5) node[anchor=north] {\tiny{$|s|^2=\frac{M_a^2s_0^2}{m_a^2-D_{a,b}}+D_{a,b}$}};
\draw (-0.03,0.43)--(0.03,0.43);
\draw (0,0.5) node[anchor=west] {$m_a$};
\end{tikzpicture}
\end{center}
\caption{\small Since the $S$-spectrum is axially symmetric we only visualize its intersection with the plane $\mathbb{C}_J$ for some arbitrary $J\in\mathbb{S}$. Here it is the two lines given by $|s|^2=\frac{M_a^2s_0^2}{m_a^2-D_{a,b}}+D_{a,b}$ from the condition \eqref{Eq_Assumption_Robin}, which separate the dark region, only consisting of points in the $S$-resolvent set, from the white region, (in general strictly) containing $\sigma_S(T)$. The gap between the two regions is of size $2\sqrt{D_{a,b}}$ and comes on the one hand from the negative part $\min\{b,0\}$ of the Robin parameter and also from the upper bound $M_a'$ of the derivative of the coefficients. The picture also visualizes the level lines of the coefficients $\tau(s):=2\big(|s|^2+m_a^2-2D_{a,b}-\sqrt{(|s|^2-m_a^2)^2+4M_a^2s_0^2}\big)^{-1}$ in the estimate \eqref{Eq_Estimate_Robin}. Here, dark colors indicate large values, light colors indicate small values. In particular, one sees that $\tau$ tends to infinity as $s$ approaches the boundary of the dark region. Also the decay of $\tau$ for $|s|\to\infty$ can be seen.}
\end{figure}

\begin{proof}
For the proof of the weak solvability \eqref{Eq_Solution_Robin}, we will verify that the form $q_s^\mathsmaller{R}$ satisfies the assumptions of the Lemma~\ref{lem_Lax_Milgram}. We start by estimating the third term in \eqref{Eq_qs_Robin} by
\begin{align*}
|\langle b\,\tau_\mathsmaller{D}u,\tau_\mathsmaller{D}v\rangle_{L^2(\partial\Omega)}|&\leq 2^{\frac{n}{2}}\Vert b\Vert_{L^\infty(\partial\Omega)}\Vert\tau_\mathsmaller{D}u
\Vert_{L^2(\partial\Omega)}\Vert\tau_\mathsmaller{D}v\Vert_{L^2(\partial\Omega)} \\
&\leq 2^{\frac{n}{2}}\Vert b\Vert_{L^\infty(\partial\Omega)}\Vert\tau_\mathsmaller{D}\Vert^2\Vert u\Vert_{H^1}\Vert v\Vert_{H^1}.
\end{align*}
Moreover, similar to \eqref{Eq_Dirichlet_5}, we can use the Sobolev inequality \eqref{Eq_Sobolev_inequality_H1}, to estimate
\begin{equation}\label{Eq_Robin_4}
\sum_{i=1}^n\Vert B_ie_iv\Vert_{L^2}^2\leq\sum_{i,j=1}^n\Big\Vert a_j\frac{\partial a_i}{\partial x_j}\Big\Vert_{L^n}^2\Vert v\Vert_{L^{\frac{2n}{n-2}}}^2\leq C_S^2\sum_{i,j=1}^n\Big\Vert a_j\frac{\partial a_i}{\partial x_j}\Big\Vert_{L^n}^2\Vert v\Vert_{H^1}^2\leq C_S^2M_a'^2\Vert v\Vert^2_{H^1}.
\end{equation}
Equipped with these two estimates, we arrive in the same way as in \eqref{Eq_Dirichlet_11} with the boundedness
\begin{equation}\label{Eq_Robin_3}
|q_s^\mathsmaller{R}(u,v)|\leq 2^{\frac{n}{2}}\big(M_a^2+2M_a|s_0|+C_SM_a'+|s|^2+\Vert b\Vert_{L^\infty(\partial\Omega)}\Vert\tau_\mathsmaller{D}\Vert^2\big)\Vert u\Vert_{H^1}\Vert v\Vert_{H^1}.
\end{equation}
For the coercivity of the scalar part $\Sc q_s^\mathsmaller{R}$, we start with the estimate
\begin{align*}
\Sc\langle b\,\tau_\mathsmaller{D}u,\tau_\mathsmaller{D}u\rangle_{L^2(\partial\Omega)}&=\int_{\partial\Omega}b\,|\tau_\mathsmaller{D}u|^2d\sigma\geq\int_{\partial\Omega}\min\{b,0\}|\tau_\mathsmaller{D}u|^2d\sigma \\
&\geq-\Vert\hspace{-0.05cm}\min\{b,0\}\Vert_{L^\infty(\partial\Omega)}\Vert\tau_\mathsmaller{D}u\Vert_{L^2(\partial\Omega)}^2\geq-\Vert\hspace{-0.05cm}\min\{b,0\}\Vert_{L^\infty(\partial\Omega)}\Vert\tau_\mathsmaller{D}\Vert^2\Vert u\Vert_{H^1}^2.
\end{align*}
Using also the estimate \eqref{Eq_Robin_4}, we end up with the coercivity
\begin{align*}
\Sc q_s^\mathsmaller{R}(u,u)&\geq\sum_{i=1}^n\bigg(m_a^2\Big\Vert\frac{\partial u}{\partial x_i}\Big\Vert_{L^2}^2-2M_a|s_0|\Big\Vert\frac{\partial u}{\partial x_i}\Big\Vert_{L^2}\Vert u\Vert_{L^2}-\Big\Vert\frac{\partial u}{\partial x_i}\Big\Vert_{L^2}\Vert B_ie_iu\Vert_{L^2}\bigg) \\
&\quad+|s|^2\Vert u\Vert_{L^2}^2+\Sc\langle b\,\tau_\mathsmaller{D}u,\tau_\mathsmaller{D}u\rangle_{L^2(\partial\Omega)} \\
&\geq m_a^2\Vert u\Vert_D^2-2M_a|s_0|\Vert u\Vert_D\Vert u\Vert_{L^2}-C_SM_a'\Vert u\Vert_D\Vert u\Vert_{H^1} \\
&\quad+|s|^2\Vert u\Vert_{L^2}^2-\Vert\hspace{-0.05cm}\min\{b,0\}\Vert_{L^\infty(\partial\Omega)}\Vert\tau_\mathsmaller{D}\Vert^2\Vert u\Vert_{H^1}^2 \\
&\geq m_a^2\Vert u\Vert_D^2-2M_a|s_0|\Vert u\Vert_D\Vert u\Vert_{L^2}+|s|^2\Vert u\Vert_{L^2}^2-D_{a,b}\Vert u\Vert_{H^1}^2.
\end{align*}
For any $\delta>0$ we now use Young's inequality
\begin{equation*}
\Vert u\Vert_D\Vert u\Vert_{L^2}\leq\frac{1}{2\delta}\Vert u\Vert_D^2+\frac{\delta}{2}\Vert u\Vert_{L^2}^2,
\end{equation*}
to further estimate this lower bound by
\begin{equation}\label{Eq_Robin_2}
\Sc q_s^\mathsmaller{R}(u,u)\geq\Big(m_a^2-\frac{M_a|s_0|}{\delta}\Big)\Vert u\Vert_D^2+(|s|^2-\delta M_a|s_0|)\Vert u\Vert_{L^2}^2-D_{a,b}\Vert u\Vert_{H^1}^2.
\end{equation}
The optimal choice on $\delta$ to get an $H^1$-lower bound, is when the two brackets in \eqref{Eq_Robin_2}, i.e. for
\begin{equation*}
\delta=\frac{|s|^2-m_a^2+\sqrt{(|s|^2-m_a^2)^2+4M_a^2s_0^2}}{2M_a|s_0|}.
\end{equation*}
Clearly, there is $\delta>0$ and if we plug it into \eqref{Eq_Robin_2}, we get
\begin{equation}\label{Eq_Robin_5}
\Sc q_s^\mathsmaller{R}(u,u)\geq\frac{1}{2}\Big(|s|^2+m_a^2-2D_{a,b}-\sqrt{(|s|^2-m_a^2)^2+4M_a^2s_0^2}\Big)\Vert u\Vert_{H^1}^2.
\end{equation}
This bracket term is positive if and only if
\begin{align}
&|s|^2+m_a^2-2D_{a,b}-\sqrt{(|s|^2-m_a^2)^2+4M_a^2s_0^2}>0 \notag \\
&\hspace{2cm}\Leftrightarrow\quad\big(|s|^2+m_a^2-2D_{a,b}\big)^2>(|s|^2-m_a^2)^2+4M_a^2s_0^2\quad\&\quad|s|^2+m_a^2>2D_{a,b} \notag \\
&\hspace{2cm}\Leftrightarrow\quad|s|^2>\frac{M_a^2s_0^2}{m_a^2-D_{a,b}}+D_{a,b}\quad\&\quad|s|^2+m_a^2>2D_{a,b} \notag \\
&\hspace{2cm}\Leftrightarrow\quad|s|^2>\frac{M_a^2s_0^2}{m_a^2-D_{a,b}}+D_{a,b}, \label{Eq_Robin_1}
\end{align}
where the last equivalence follows from $\frac{M_a^2s_0^2}{m_a^2-D_{a,b}}+D_{a,b}\geq 2D_{a,b}-m_a^2$, which is true by \eqref{Eq_Assumption_Robin}.  However, the condition \eqref{Eq_Robin_1} is exactly the stated assumption \eqref{Eq_Assumption_Robin}, which means that we have proven the coercivity of the Robin form $q_s^\mathsmaller{R}$. Let now $f\in L^2(\Omega)$, and consider the functional $\varphi_f(v):=\langle f,v\rangle_{L^2}$, $v\in H^1(\Omega)$, which is bounded in $H^1(\Omega)$ by the same estimate as in \eqref{Eq_Dirichlet_3}. Lemma~\ref{lem_Lax_Milgram} then states the existence of a unique $u_f\in H^1(\Omega)$, which satisfies
\begin{equation*}
q_s^\mathsmaller{R}(u_f,v)=\varphi_f(v)=\langle f,v\rangle_{L^2},\qquad\text{for all }v\in H^1(\Omega).
\end{equation*}
Testing this equation with $v=u_f$, using the coercivity estimate \eqref{Eq_Robin_5} as well as the analogous estimate \eqref{Eq_Dirichlet_7} for $q_s^\mathsmaller{R}$, gives the $H^1$-norm estimate of the solution
\begin{equation*}
\Vert u_f\Vert_{H^1}^2\leq\frac{2\Vert f\Vert_{L^2}\Vert u_f\Vert_{L^2}}{|s|^2+m_a^2-2D_{a,b}-\sqrt{(|s|^2-m_a^2)^2+4M_a^2s_0^2}},
\end{equation*}
which is exactly \eqref{Eq_Estimate_Robin} if we estimate $\Vert u_f\Vert_{L^2}\leq\Vert u_f\Vert_{H^1}$ on the right hand side.
\end{proof}

If the coefficients $a_1,\dots,a_n$ of the operator $T$ in \eqref{Eq_T} are constant, and if the Robin parameter $b\geq 0$ is nonnegative, which in particular includes Neumann boundary conditions $b\equiv 0$, the results of Theorem~\ref{thm_Robin} can be refined. In particular the gap at $s=0$ closes, see Figure~3. Also observe that that the norm estimate \eqref{Eq_Estimate_Robin} of the solution can be improved.

\begin{cor}
Consider $b\in L^\infty(\partial\Omega)$ with $b\geq 0$ and constant coefficients $a_1,\dots,a_n\in(0,\infty)$ in \eqref{Eq_T}. In this case we have $m_a=\min\{a_1,\dots,a_n\}$, $M_a=\max\{a_1,\dots,a_n\}$ and $M_a'=0$, for the values of \eqref{Eq_Coefficient_bounds}. Let us consider $s\in\mathbb{R}^{n+1}$ with
\begin{equation}\label{Eq_Assumption_Robin_constant}
|s|>\frac{M_a}{m_a}|s_0|.
\end{equation}
Then, for every $f\in L^2(\Omega)$, there exists a unique $u_f\in H^1(\Omega)$ such that
\begin{equation}\label{Eq_Solution_Robin_constant}
q_s^\mathsmaller{R}(u_f,v)=\langle f,v\rangle_{L^2},\qquad\text{for all }v\in H^1(\Omega).
\end{equation}
Moreover, this solution satisfies the bounds
\begin{equation}\label{Eq_Estimate_Robin_constant}
\Vert u_f\Vert_{L^2}\leq\frac{\Vert f\Vert_{L^2}}{|s|^2-\frac{M_a^2s_0^2}{m_a^2}}\qquad\text{and}\qquad\Vert u_f\Vert_D\leq\frac{|s|\,\Vert f\Vert_{L^2}}{m_a\big(|s|^2-\frac{M_a^2s_0^2}{m_a^2}\big)}.
\end{equation}
The visualization of these estimates is the same as in Figure~1 with $K_a=\frac{M_a}{m_a}$.
\end{cor}

\begin{proof}
Since we assumed $b\geq 0$ and $a_1,\dots,a_n\in(0,\infty)$ being positive and constant, it is clear that $\Vert\hspace{-0.05cm}\min\{b,0\}\Vert_{L^\infty(\partial\Omega)}=0$ and $M_a'=0$. Hence also the constant $D_{a,b}=0$ in Theorem~\ref{thm_Robin} vanishes and the condition \eqref{Eq_Assumption_Robin} turns into \eqref{Eq_Assumption_Robin_constant}. This immediately yields the unique solvability \eqref{Eq_Solution_Robin_constant} from Theorem~\ref{thm_Robin}. Moreover, to verify the estimates \eqref{Eq_Estimate_Robin_constant}, we note that the coercivity estimate \eqref{Eq_Robin_2} looks like
\begin{equation}\label{Eq_Robin_constant_1}
\Sc q_s^\mathsmaller{R}(u_f,u_f)\geq\Big(m_a^2-\frac{M_a|s_0|}{\delta}\Big)\Vert u_f\Vert_D^2+(|s|^2-\delta M_a|s_0|)\Vert u_f\Vert_{L^2}^2,
\end{equation}
for every $\delta>0$. For the $L^2$-estimate in \eqref{Eq_Estimate_Robin_constant}, we choose $\delta=\frac{M_a|s_0|}{m_a^2}$, which turns \eqref{Eq_Robin_constant_1} into
\begin{equation*}
\Sc q_s^\mathsmaller{R}(u_f,u_f)\geq\Big(|s|^2-\frac{M_a^2s_0^2}{m_a^2}\Big)\Vert u_f\Vert_{L^2}^2.
\end{equation*}
Together with the analog estimate \eqref{Eq_Dirichlet_7} for $q_s^\mathsmaller{R}$, this leads to the stated bound
\begin{equation*}
\Vert u_f\Vert_{L^2}^2\leq\frac{\Sc q_s^\mathsmaller{R}(u_f,u_f)}{|s|^2-\frac{M_a^2s_0^2}{m_a^2}}\leq\frac{\Vert f\Vert_{L^2}\Vert u_f\Vert_{L^2}}{|s|^2-\frac{M_a^2s_0^2}{m_a^2}}.
\end{equation*}
For the $\Vert\cdot\Vert_D$-estimate in \eqref{Eq_Estimate_Robin_constant}, we choose $\delta=\frac{|s|^2}{M_a|s_0|}$, which turns \eqref{Eq_Robin_constant_1} into
\begin{equation*}
\Sc q_s^\mathsmaller{R}(u_f,u_f)\geq\frac{m_a^2}{|s|^2}\Big(|s|^2-\frac{M_a^2s_0^2}{m_a^2}\Big)\Vert u_f\Vert_D^2.
\end{equation*}
Together with the analog estimate \eqref{Eq_Dirichlet_7} for $q_s^\mathsmaller{R}$, this leads to the second bound in \eqref{Eq_Estimate_Robin_constant}, namely
\begin{equation*}
\Vert u_f\Vert_D^2\leq\frac{|s|^2\Sc q_s^\mathsmaller{R}(u_f,u_f)}{m_a^2\big(|s|^2-\frac{M_a^2s_0^2}{m_a^2}\big)}\leq\frac{|s|^2\Vert f\Vert_{L^2}\Vert u_f\Vert_{L^2}}{m_a^2\big(|s|^2-\frac{M_a^2s_0^2}{m_a^2}\big)}\leq\frac{|s|^2\Vert f\Vert_{L^2}^2}{m_a^2\big(|s|^2-\frac{M_a^2s_0^2}{m_a^2}\big)^2},
\end{equation*}
where in the last inequality we used the already calculated $L^2$-estimate in \eqref{Eq_Estimate_Robin_constant}.
\end{proof}

\section{An auxiliary maximization problem}\label{sec_Maximization_problem}

In this section we collect an auxiliary maximization problem that we have used in the proofs of the main theorem.

\begin{lem}\label{lem_Maximization_problem}
For every $x,y>0$, there holds the equivalence
\begin{equation}\label{Eq_Equivalence}
\exists\alpha\in(0,1): (x-2)y+2\alpha y-\alpha^2>0 \Leftrightarrow y>\begin{cases} \frac{1}{x}, & \text{if }x\leq 1, \\ 2-x, & \text{if }x\geq 1. \end{cases}
\end{equation}
More precisely, if the two equivalent conditions in \eqref{Eq_Equivalence} are satisfied, then $x+y>2$, and
\begin{equation}\label{Eq_Interval_of_positivity}
\Set{\alpha\in(0,1) | (x-2)y+2\alpha y-\alpha^2>0}=(\alpha_0,\alpha_1),
\end{equation}
where $\alpha_0<\alpha_1\in[0,1]$ are given by
\begin{equation}\label{Eq_alpha0_alpha1}
\alpha_0=\begin{cases} y-\sqrt{y(x+y-2)}, & \text{if }x\leq 2, \\ 0, & \text{if }x\geq 2, \end{cases} \quad\text{and}\quad \alpha_1:=\begin{cases} y+\sqrt{y(x+y-2)}, & \text{if }y\leq\frac{1}{x}, \\ 1, & \text{if }y\geq\frac{1}{x}. \end{cases}
\end{equation}
For every $\alpha\in(\alpha_0,\alpha_1)$ we now have $x-2+2\alpha>0$ and there exists a gap between
\begin{equation}\label{Eq_delta0_delta1}
\delta_0(\alpha):=\frac{\alpha}{x-2+2\alpha}<\frac{y}{\alpha}=:\delta_1(\alpha).
\end{equation}
Moreover, for every $0<C\leq 1$ we have the two maximization problems

\begin{minipage}{0.3\textwidth}
\begin{center}
\begin{tikzpicture}[scale=3.5]
\fill[black!30] (1.05,0.8)--(1.05,0)--(0.7,0)--(0.35,0.35) arc (-135:-165:1);
\draw[->] (-0.1,0)--(1.15,0) node[anchor=south] {\tiny{$x$}};
\draw[->] (0,-0.1)--(0,0.85) node[anchor=east] {\tiny{$y$}};
\draw[thick] (1.05,0)--(0.7,0)--(0.35,0.35) arc (-135:-165:1);
\draw[dashed] (0.35,0.35)--(1.05,0.35);
\draw[dashed] (0.35,0.355) arc (-135:-90.5:1);
\draw[dashed] (0.35,0.35)--(0.35,0.8);
\draw[dashed] (0.7,0)--(0.7,0.8);
\draw (0.26,0.43) node[anchor=east] {\tiny{$\frac{1}{x}$}};
\draw (0.35,0.02)--(0.35,-0.02) node[anchor=north] {\tiny{$1$}};
\draw (0.7,0.02)--(0.7,-0.02) node[anchor=north] {\tiny{$2$}};
\draw (0.02,0.35)--(-0.02,0.35) node[anchor=east] {\tiny{$1$}};
\end{tikzpicture}
\end{center}
\end{minipage}
\begin{minipage}{0.69\textwidth}
\begin{subequations}\label{Eq_Maximization_problem}
\begin{align}
\sup\limits_{\alpha\in(\alpha_0,\alpha_1)}\alpha\Big(2-\frac{\alpha}{y}\Big)&=\begin{cases} y, & \text{if }y\leq 1, \\ 2-\frac{1}{y}, & \text{if }y\geq 1, \end{cases} \label{Eq_Maximization_problem1} \\
\sup\limits_{\stackrel{\alpha\in(\alpha_0,\alpha_1)}{\mathsmaller{\delta\in(\delta_0(\alpha),\delta_1(\alpha))}}}\alpha\Big(2-\frac{1}{\delta}-\delta\Big)&=\begin{cases} 0, & \text{if }x\geq 1, \\ 2-x-\frac{1}{x}, & \text{if }x\leq 1. \end{cases} \label{Eq_Maximization_problem2}
\end{align}
\end{subequations}
\end{minipage}
\end{lem}

\begin{proof}
Let us define the polynomial
\begin{equation}\label{Eq_Equivalence_11}
p(\alpha):=(x-2)y+2\alpha y-\alpha^2,\qquad\alpha\in(0,1).
\end{equation}
In the \textit{first step} we calculate the maximal value $\sup_{\alpha\in(0,1)}p(\alpha)$. To do so, we first note, that the (possible) zero of the derivative
\begin{equation}\label{Eq_Equivalence_1}
p'(\alpha)=2(y-\alpha),
\end{equation}
is given by $\alpha=y$. However, since we only consider the polynomial on the interval $\alpha\in(0,1)$, the maximum value is attained at $\alpha=y$ if $y<1$ and at the right endpoint $\alpha\to 1$ if $y\geq 1$. This then gives the maximal value
\begin{equation}\label{Eq_Equivalence_5}
\sup\limits_{\alpha\in(0,1)}p(\alpha)=\begin{cases} p(y), & \text{if }y<1, \\ p(1), & \text{if }y\geq 1, \end{cases}=\begin{cases} y(x+y-2), & \text{if }y\leq 1, \\ xy-1, & \text{if }y\geq 1. \end{cases}
\end{equation}
In the \textit{second step} we will prove the equivalence \eqref{Eq_Equivalence}. Using the already calculated maximal value \eqref{Eq_Equivalence_5} allows us to rewrite \eqref{Eq_Equivalence} as
\begin{equation}\label{Eq_Equivalence_2}
0<\begin{cases} y(x+y-2), & \text{if }y\leq 1, \\ xy-1, & \text{if }y\geq 1, \end{cases} \Leftrightarrow y>\begin{cases} \frac{1}{x}, & \text{if }x\leq 1, \\ 2-x, & \text{if }x\geq 1. \end{cases}
\end{equation}
The proof of this equivalence will be done for the two cases of the left hand side separately. \medskip

$\circ$\;\;If $y\leq 1$, the equivalence looks like
\begin{equation*}
0<x+y-2\Leftrightarrow y>\begin{cases} \frac{1}{x}, & \text{if }x\leq 1, \\ 2-x, & \text{if }x\geq 1. \end{cases}
\end{equation*}
It is true, because for $x\geq 1$ both inequalities are the same, and for $x\leq 1$ both inequalities are never satisfied. \medskip

$\circ$\;\;If $y>1$, the equivalence \eqref{Eq_Equivalence_2} looks like
\begin{equation*}
0<xy-1\Leftrightarrow y>\begin{cases} \frac{1}{x}, & \text{if }x\leq 1, \\ 2-x, & \text{if }x\geq 1. \end{cases}
\end{equation*}
It is true, because for $x\leq 1$ both inequalities are the same, and for $x\geq 1$ both inequalities are always satisfied. \medskip

In the \textit{third step} we prove \eqref{Eq_Interval_of_positivity}. The zeros of $p$, if we forget about the constraint $\alpha\in(0,1)$ for the moment, are given by
\begin{equation}\label{Eq_Equivalence_4}
\alpha_\pm=y\pm\sqrt{y(x+y-2)}.
\end{equation}
Since we assumed that \eqref{Eq_Equivalence} is satisfied, there exists at least one point $\alpha\in(0,1)$, for which $p(\alpha)>0$. Since moreover $\lim_{\alpha\to\pm\infty}p(\alpha)=-\infty$, the zeros \eqref{Eq_Equivalence_4} have to be real valued, i.e. $x+y>2$ is satisfied and the interval where $p$ is positive is exactly between those zeros,
\begin{equation*}
\Set{\alpha\in\mathbb{R} | p(\alpha)>0}=(\alpha_-,\alpha_+).
\end{equation*}
Restricting now again to the interval $\alpha\in(0,1)$, we get
\begin{equation*}
\Set{\alpha\in(0,1) | p(\alpha)>0}=\big(\max\{0,\alpha_-\},\min\{1,\alpha_+\}\big).
\end{equation*}
It is left to check when we have $\alpha_-\geq 0$ and $\alpha_+\leq 1$ respectively. The condition $\alpha_-\geq 0$ is equivalent to
\begin{equation}\label{Eq_Equivalence_8}
\alpha_-\geq 0\quad\Leftrightarrow\quad y\geq\sqrt{y(x+y-2)}\quad\Leftrightarrow\quad\sqrt{y}\geq\sqrt{x+y-2}\quad\Leftrightarrow\quad x\leq 2,
\end{equation}
and the condition $\alpha_+\leq 1$ is equivalent to
\begin{align}
\alpha_+\leq 1\quad&\Leftrightarrow\quad y+\sqrt{y(x+y-2)}\leq 1 \notag \\
&\Leftrightarrow\quad 1-y\geq 0\quad\text{and}\quad y(x+y-2)\leq(1-y)^2 \notag \\
&\Leftrightarrow\quad y\leq 1\quad\text{and}\quad xy\leq 1 \notag \\
&\Leftrightarrow\quad y\leq\frac{1}{x}, \label{Eq_Equivalence_9}
\end{align}
where in the last equivalence we used that the combination $y>1$ and $y\leq\frac{1}{x}$ is not allowed, since in this case we have $1<y\leq\frac{1}{x}$, consequently $x<1$, which is a contradiction to the assumed \eqref{Eq_Equivalence_2}, see also the shaded region in \eqref{Eq_Maximization_problem}. The two equivalences \eqref{Eq_Equivalence_8} and \eqref{Eq_Equivalence_9} now  show that $\max\{0,\alpha_-\}=\alpha_0$ and $\min\{1,\alpha_+\}=\alpha_1$ coincide with the values in \eqref{Eq_alpha0_alpha1}. Hence \eqref{Eq_Interval_of_positivity} is proven. \medskip

In the \textit{fourth step}, we will solve the first maximization problem \eqref{Eq_Maximization_problem1}, i.e. we will maximize
\begin{equation*}
f(\alpha):=\alpha\Big(2-\frac{\alpha}{y}\Big),\qquad\alpha\in(\alpha_0,\alpha_1).
\end{equation*}
The zero of the derivative
\begin{equation}\label{Eq_Equivalence_10}
f'(\alpha)=2\Big(1-\frac{\alpha}{y}\Big),
\end{equation}
is given by $\alpha=y$. We now have to check if this value $y$ is contained in the interval $(\alpha_0,\alpha_1)$ of allowed values. First,
\begin{equation}\label{Eq_Equivalence_6}
y>\begin{cases} y-\sqrt{y(x+y-2)}, & \text{if }x\leq 2, \\ 0, & \text{if }x\geq 2, \end{cases}=\alpha_0,\qquad\text{is always satisfied.}
\end{equation}
The second condition $y<\alpha_1$ can be characterized by
\begin{equation}\label{Eq_Equivalence_7}
\begin{split}
y<\alpha_1\quad&\Leftrightarrow\quad y<\begin{cases} y+\sqrt{y(x+y-2)}, & \text{if }y\leq\frac{1}{x} \\ 1, & \text{if }y\geq\frac{1}{x} \end{cases} \\
&\Leftrightarrow\quad\Big(y\leq\frac{1}{x}\Big)\quad\text{or}\quad\Big(y\geq\frac{1}{x}\quad\text{and}\quad y<1\Big) \\
&\Leftrightarrow\quad y<1,
\end{split}
\end{equation}
where in the last equivalence we used that the combination $y>1$ and $y\leq\frac{1}{x}$ is not allowed, similar to the argument below equation \eqref{Eq_Equivalence_9}. Hence, the maximal value of $f$ is attained at $\alpha=y$ if $y<1$ and at the right endpoint $\alpha\to\alpha_1=1$ if $y\geq 1$. This means, we get
\begin{equation*}
\sup\limits_{\alpha\in(\alpha_0,\alpha_1)}f(\alpha)=\begin{cases} f(y), & \text{if }y<1, \\ f(1), & \text{if }y\geq 1, \end{cases}=\begin{cases} y, & \text{if }y\leq 1, \\ 2-\frac{1}{y}, & \text{if }y\geq 1. \end{cases}
\end{equation*}
In the \textit{fifth step} we first note that for every $\alpha\in(\alpha_0,\alpha_1)$, there holds
\begin{equation*}
x-2+2\alpha>x-2+2\alpha_0=\begin{cases} x-2+2y-2\sqrt{y(x+y-2)}, & \text{if }x\leq 2, \\ x-2, & \text{if }x\geq 2, \end{cases} \geq 0,
\end{equation*}
where the right hand side is nonnegative because $a+b-2\sqrt{ab}\geq 0$, for every $a,b\geq 0$. Moreover, there holds
\begin{equation*}
\delta_0(\alpha):=\frac{\alpha}{x-2+2\alpha}<\frac{y}{\alpha}=:\delta_1(\alpha),\qquad\alpha\in(\alpha_0,\alpha_1),
\end{equation*}
which is a simple rearrangement of the inequality $(x-2)y+2\alpha y-\alpha^2>0$ in \eqref{Eq_Interval_of_positivity}. Now, we will solve the second maximization problem \eqref{Eq_Maximization_problem2}, i.e. we want to maximize the function
\begin{equation*}
g(\alpha,\delta):=\alpha\Big(2-\frac{1}{\delta}-\delta\Big),\qquad\alpha\in(\alpha_0,\alpha_1),\,\delta\in(\delta_0(\alpha),\delta_1(\alpha)).
\end{equation*}
First, we will fix $\alpha\in(\alpha_0,\alpha_1)$ and maximize $g(\alpha,\cdot\,)$ with respect to $\delta$. The only positive zero of the partial derivative
\begin{equation*}
\frac{\partial g}{\partial\delta}(\alpha,\delta)=\alpha\Big(\frac{1}{\delta^2}-1\Big),
\end{equation*}
is given by $\delta=1$. The maximum value of $g(\alpha,\cdot\,)$ now depends on the location of this zero with respect to the interval $(\delta_0(\alpha),\delta_1(\alpha))$ of allowed $\delta$-values. If $1\leq\delta_0(\alpha)$, the maximum is attained the left endpoint $\delta\to\delta_0(\alpha)^+$, if $1\in(\delta_0(\alpha),\delta_1(\alpha))$ is an allowed value, the maximum is attained at exactly this point $\delta=1$, and if $1\geq\delta_1(\alpha)$, the maximum value is attained at the right endpoint $\delta\to\delta_1(\alpha)$. I.e., we get
\begin{align*}
h(\alpha):=\sup\limits_{\delta\in(\delta_0(\alpha),\delta_1(\alpha))}g(\alpha,\delta)&=\begin{cases} g(\alpha,\delta_0(\alpha)^+) & \text{if }1\leq\delta_0(\alpha), \\ g(\alpha,\frac{1}{C}), & \text{if }\delta_0(\alpha)<1<\delta_1(\alpha), \\ g(\alpha,\delta_1(\alpha)), & \text{if }1\geq\delta_1(\alpha), \end{cases} \\
&=\begin{cases} 2-x-\frac{\alpha^2}{x-2+2\alpha}, & \text{if }\alpha\leq 2-x, \\ 0, & \text{if }2-x\leq\alpha\leq y, \\ 2\alpha-\frac{\alpha^2}{y}-y, & \text{if }\alpha\geq y. \end{cases}
\end{align*}
Note, that we were allowed to replace the $<$-inequalities by $\leq$-inequalities in the second case since $h(\alpha)$ is continuous at $\alpha=2-x$ and $\alpha=y$. Next, we want to maximize this function $h$ with respect to the parameter $\alpha$. Therefore, we calculate its derivative
\begin{equation}\label{Eq_Equivalence_3}
h'(\alpha)=\begin{cases} \frac{2\alpha(2-x-\alpha)}{(x-2+2\alpha)^2}, & \text{if }\alpha\leq 2-x, \\ 0, & \text{if }2-x\leq\alpha\leq y, \\ 2(1-\frac{\alpha}{y}), & \text{if }\alpha\geq y. \end{cases}
\end{equation}
In the first two cases $\alpha\leq 2-x$ and $2-x\leq\alpha\leq y$ in \eqref{Eq_Equivalence_3}, we obviously have a nonnegative derivative $h'(\alpha)\geq 0$. Only in the third case $\alpha\geq y$, there exists the possible zero $\alpha_{\max}=y$ of the derivative. It is already calculated in \eqref{Eq_Equivalence_6} and \eqref{Eq_Equivalence_7}, that
\begin{equation*}
\alpha_{\max}\in(\alpha_0,\alpha_1)\Leftrightarrow y<1\qquad\text{and}\qquad \alpha_{\max}\geq\alpha_1\Leftrightarrow y\geq 1.
\end{equation*}
Hence, we conclude that for $y\geq 1$ the function $h(\alpha)$ is monotone increasing on the whole interval $(\alpha_0,\alpha_1)$, which means that its supremum is attained at the right endpoint $\alpha\to\alpha_1=1$. In the case $y<1$ on the other hand, the maximal value is attained at the interior point $\alpha=\alpha_{\max}=y$. The solution of the maximization problem is then given by
\begin{align*}
\sup\limits_{\alpha\in(\alpha_0,\alpha_1)}h(\alpha)&=\begin{cases} h(1), & \text{if }y\geq 1, \\ h(y), & \text{if }y<1, \end{cases}=\begin{cases} 2-x-\frac{1}{x}, & \text{if }y\geq 1\;\;\&\;\;x\leq 1, \\ 0, & \text{if }y\geq 1\;\;\&\;\;x\geq 1, \\ 2-\frac{1}{y}-y, & \text{if }y\geq 1\;\;\&\;\;1\geq y, \\ 0, & \text{if }y<1, \end{cases} \\
&=\begin{cases} 2-x-\frac{1}{x}, & \text{if }x\leq 1, \\ 0, & \text{if }x\geq 1. \end{cases} \qedhere
\end{align*}
\end{proof}

\section{Concluding remarks}\label{sec_Concluding_remarks}

We conclude this paper with some considerations that contextualize the estimates we have derived within the framework of the $H^\infty$-functional calculus, which manifests in various settings. In the realm of hypercomplex settings, depending on the spectral theory under consideration, this calculus takes on several distinct formulations and the estimates for sectorial and bi-sectorial operators are crucial for establishing $H^\infty$-functional calculi. \medskip

We recall that the $H^\infty$-functional calculus for complex operators on Banach spaces, introduced by A. McIntosh \cite{McI1} and further investigated in \cite{MC10,MC97,MC06,MC98}, holds significant relevance across multiple disciplines. For more information and the developments one can consult the books \cite{Haase,HYTONBOOK1,HYTONBOOK2}. \medskip

Regarding the spectral theory on the $S$-spectrum \cite{ACS2016,AlpayColSab2020,FJBOOK,CGK,ColomboSabadiniStruppa2011,ColSab2006,JONAMEM,JONADIRECT}, which has its roots in the theory of slice hyperholomorphic functions, the $H^\infty$-functional calculus has been considered in the papers \cite{ACQS2016,CGdiffusion2018}. \medskip

The spectral theory concerning the monogenic spectrum, introduced in \cite{JM}, relies on monogenic functions \cite{DSS}. Further exploration of this topic, including the $H^\infty$-functional calculus within this framework can be found in the books \cite{JBOOK,TAOBOOK}. \medskip

In recent times, a novel branch of spectral theory concerning the $S$-spectrum  has emerged in connection with the factorizations of the second mapping in the Fueter-Sce extension theorem, see \cite{ColSabStrupSce,Fueter,TaoQian1,Sce}. This branch, referred to as the fine structures on the $S$-spectrum, revolves around sets of functions characterized by integral representations, from which the corresponding functional calculi are defined. The development of this new theory and its associated functional calculi is elaborated in the papers \cite{BANJOUR,CDPS1,Fivedim,Polyf1,Polyf2}, while recent investigations into the $H^\infty$-functional calculus linked with the functional calculi of the fine structures are studied in \cite{MILANJPETER,MPS23}.

\section*{Declarations and statements}

{\bf Data availability}. The research in this paper does not imply use of data.

{\bf Conflict of interest}. The authors declare that there is no conflict of interest.


\begin{thebibliography}{99}

\bibitem{ACK} D. Alpay, F. Colombo, D.P. Kimsey: \textit{The spectral theorem for quaternionic unbounded normal operators based on the $S$-spectrum}. J. Math. Phys. \textbf{57}(2), 023503, 27 (2016).

\bibitem{ACS2016} D. Alpay, F. Colombo, I. Sabadini: \textit{Slice hyperholomorphic Schur analysis}. Volume 256 of Operator Theory: Advances and Applications, Basel, Birkhäuser/Springer (2016).

\bibitem{AlpayColSab2020} D. Alpay, F. Colombo, I. Sabadini: \textit{Quaternionic de Branges spaces and characteristic operator function}. SpringerBriefs in Mathematics, Springer, Cham (2020).

\bibitem{ACQS2016} D. Alpay, F. Colombo, T. Qian, I. Sabadini: \textit{The $H^\infty$ functional calculus based on the $S$-spectrum for quaternionic operators and for $n$-tuples of noncommuting operators}. J. Funct. Anal. \textbf{271}(6) (2016) 1544--1584.

\bibitem{MC10} P. Auscher, A. Axelsson, A. McIntosh: \textit{On a quadratic estimate related to the Kato conjecture and boundary value problems}. Harmonic analysis and partial differential equations, Contemp. Math., 505, Amer. Math. Soc., Providence, RI (2010) 105-129.

\bibitem{MC97} P. Auscher, A. McIntosh, A. Nahmod. \textit{Holomorphic functional calculi of operators, quadratic estimates and interpolation}. Indiana Univ. Math. J. \textbf{46} (1997) 375--403.

\bibitem{MC06} A. Axelsson, S. Keith, A. McIntosh: \textit{Quadratic estimates and functional calculi of perturbed Dirac operators}. Invent. Math. \textbf{163} (2006) 455--497.

\bibitem{BARACCO} L. Baracco, F. Colombo, M.M. Peloso, S. Pinton: \textit{Fractional powers of higher-order vector operators on bounded and unbounded domains}. Proc. Edinb. Math. Soc. \textbf{4}(2) 65 (2022) 912-937.

\bibitem{BF} G. Birkhoff, J. von Neumann: \textit{The logic of quantum mechanics}. Ann. of Math. (2) \textbf{37}(4) (1936) 823--843.

\bibitem{B11} H. Brezies: \textit{Functional Analysis, Sobolev Spaces and Partial Differential Equations}. Springer Science+Business Media (2011).
	
\bibitem{BANJOUR} F. Colombo, A. De Martino, S. Pinton: \textit{Harmonic and polyanalytic functional calculi on the $S$-spectrum for unbounded operators}. Banach J. Math. Anal. \textbf{17}, no. 4, Paper No. 84 (2023) 41 pp.

\bibitem{CDPS1} F. Colombo, A. De Martino, S. Pinton, I. Sabadini: \textit{Axially harmonic functions and the harmonic functional calculus on the $S$-spectrum}. J. Geom. Anal. \textbf{33}(2) (2023), 54 pp.

\bibitem{Fivedim} F. Colombo, A. De Martino, S. Pinton, I. Sabadini: \textit{The fine structure of the spectral theory on the $S$-spectrum in dimension five}. J. Geom. Anal. \textbf{33}, no. 9, Paper No. 300 (2023) 73 pp.

\bibitem{ColomboDenizPinton2020} F. Colombo, D. Deniz Gonz\'alez, S. Pinton: \textit{Fractional powers of vector operators with first order boundary conditions}. J. Geom. Phys. \textbf{151}, 18 (2020) 103618.

\bibitem{ColomboDenizPinton2021} F. Colombo, D. Deniz Gonz\'alez, S. Pinton: \textit{The noncommutative fractional Fourier law in bounded and unbounded domains}. Complex Anal. Oper. Theory \textbf{15}(7), Paper No. 114, 27 (2021).

\bibitem{CGdiffusion2018} F. Colombo. J. Gantner: \textit{An application of the $S$-functional calculus to fractional diffusion processes}. Milan J. Math. \textbf{86}(2) (2018) 225--303.

\bibitem{FJBOOK} F. Colombo, J. Gantner: \textit{Quaternionic closed operators, fractional powers and fractional diffusion processes}. Operator Theory: Advances and Applications \textbf{274}, Birkhäuser/Springer, Cham (2019) viii+322.

\bibitem{CGK} F. Colombo, J. Gantner, D.P. Kimsey: \textit{Spectral theory on the $S$-spectrum for quaternionic operators}. Volume 270 of Operator Theory: Advances and Applications, Birkhäuser/Springer, Cham (2018) ix+356.

\bibitem{ADVCGKS} F. Colombo, J. Gantner, D.P. Kimsey, I. Sabadini: \textit{Universality property of the $S$-functional calculus, noncommuting matrix variables and Clifford operators}. Adv. Math. \textbf{410}, Paper No. 108719, (2022) 39 pp.

\bibitem{ColKim} F. Colombo, D.P. Kimsey: \textit{The spectral theorem for normal operators on a Clifford module}. Anal. Math. Phys., \textbf{12}(1), Paper No. 25, 92 (2022).

\bibitem{ColomboPelosoPinton2019} F. Colombo, M.M. Peloso, S. Pinton: \textit{The structure of the fractional powers of the noncommutative Fourier law}. Math. Methods Appl. Sci. \textbf{42}(18) (2019) 6259--6276.

\bibitem{MILANJPETER} F. Colombo, S. Pinton, P. Schlosser: \textit{The $H^\infty$-functional calculi for the quaternionic fine structures of Dirac type}. Milan J. Math., \textbf{92} (2024), no. 1, 73-122.

\bibitem{ColSabStrupSce} F. Colombo, I. Sabadini, D.C. Struppa: \textit{Michele Sce's works in hypercomplex analysis -- A translation with commentaries}. Birkh\"auser/Springer, Cham (2020).

\bibitem{ColomboSabadiniStruppa2011} F. Colombo, I. Sabadini, D.C. Struppa: \textit{Noncommutative functional calculus. Theory and applications of slice hyperholomorphic functions}. Volume 289 of Progress in Mathematics, Birkhäuser/Springer Basel AG, Basel (2011) vi+221 pp.

\bibitem{ColSab2006} F. Colombo, I. Sabadini, D.C. Struppa: \textit{A new functional calculus for noncommuting operators}. J. Funct. Anal. \textbf{254}(8) (2008) 2255--2274.

\bibitem{Polyf1} A. De Martino, S. Pinton: \textit{A polyanalytic functional calculus of order 2 on the $S$-spectrum}. Proc. Amer. Math. Soc. \textbf{151} (2023) 2471--2488.

\bibitem{Polyf2} A. De Martino, S. Pinton: \textit{Properties of a polyanalytic functional calculus on the $S$-spectrum}. Math. Nachr. \textbf{296} (2023) 5190-5226.

\bibitem{MPS23} A. De Martino, S. Pinton, P. Schlosser: \textit{The harmonic $H^\infty$-functional calculus based on the $S$-spectrum}.  J. Spectr. Theory. \textbf{14} (2024), no. 1, 121-162.

\bibitem{DSS} R. Delanghe, F. Sommen, V. Sou\v cek: \textit{Clifford algebra and spinor-valued functions.} Mathematics and its Applications \textbf{53}, Kluwer Academic Publishers Group, Dordrecht (1992).

\bibitem{RD} N. Dunford, J. Schwartz: \textit{Linear operators, part I: general theory}. J. Wiley and Sons (1988).

\bibitem{MC98} E. Franks, A. McIntosh: \textit{Discrete quadratic estimates and holomorphic functional calculi in Banach spaces}. Bull. Austral. Math. Soc. \textbf{58} (1998) 271--290.

\bibitem{Fueter} R. Fueter: \textit{Die Funktionentheorie der Differentialgleichungen $\Delta u=0$ und $\Delta\Delta u=0$ mit vier reellen Variablen}. Comment. Math. Helv. \textbf{7}(1) (1934) 307--330.

\bibitem{JONAMEM} J. Gantner: \textit{Operator theory on one-sided quaternion linear spaces: intrinsic S-functional calculus and spectral operators.} Mem. Amer. Math. Soc. \textbf{267}, no. 1297 (2020) iii+101 pp.

\bibitem{JONADIRECT}
J. Gantner: \textit{A direct approach to the $S$-functional calculus for closed operators}. J. Operator Theory. \textbf{77} no.2 (2017) 287-331.

\bibitem{JBOOK} B. Jefferies: \textit{Spectral properties of noncommuting operators}. Volume 1843, Lecture Notes in Mathematics. Springer-Verlag, Berlin (2004).

\bibitem{Haase} M. Haase: \textit{The functional calculus for sectorial operators}. Operator Theory: Advances and Applications \textbf{169}, Birkhäuser Verlag, Basel (2006).

\bibitem{HYTONBOOK1} T. Hytönen, J. van Neerven, M. Veraar, L. Weis: \textit{Analysis in Banach spaces. Vol. II. Probabilistic methods and operator theory}. Ergebnisse der Mathematik und ihrer Grenzgebiete. 3. Folge. A Series of Modern Surveys in Mathematics \textbf{67}, Springer, Cham (2017) xxi+616 pp.

\bibitem{HYTONBOOK2} T. Hytönen, J. van Neerven, M. Veraar, and L. Weis: \textit{Analysis in Banach spaces. Vol. I. Martingales and Littlewood-Paley theory}. Ergebnisse der Mathematik und ihrer Grenzgebiete. 3. Folge. A Series of Modern Surveys in Mathematics \textbf{63}, Springer, Cham (2016) xvi+614 pp.

\bibitem{JM} B. Jefferies, A. McIntosh, J. Picton-Warlow: \textit{The monogenic functional calculus}. Studia Math. \textbf{136} (1999) 99-119.

\bibitem{TAOBOOK} P. Li, T. Qian: \textit{Singular integrals and Fourier theory on Lipschitz boundaries}. Science Press Beijing, Beijing, Springer, Singapore (2019) xv+306 pp.

\bibitem{M87} J. Marschall: \textit{The trace of Sobolev-Slobodeckij spaces on Lipschitz domains}. Manuscripta Math. \textbf{58} (1987) 47-65.

\bibitem{McI1} A. McIntosh: \textit{Operators which have an $H^\infty$ functional calculus}. Miniconference on operator theory and partial differential equations (North Ryde, 1986), 210--231, Proc. Centre Math. Anal. Austral. Nat. Univ., 14, Austral. Nat. Univ., Canberra, (1986).

\bibitem{TaoQian1} T. Qian: \textit{Generalization of Fueter's result to $\textbf{R}^{n+1}$}. Atti Accad. Naz. Lincei Cl. Sci. Fis. Mat. Natur. Rend. Lincei (9) Mat. Appl. \textbf{8}(2) (1997) 111--117.

\bibitem{Sce} M. Sce: \textit{Osservazioni sulle serie di potenze nei moduli quadratici}. Atti Accad. Naz. Lincei Rend. Cl. Sci. Fis. Mat. Nat. \textbf{23}(8) (1957) 220--225.

\bibitem{S70} E. Stein: \textit{Singular integrals and differentiability properties of functions}. Princeton University Press, Princeton, New Yersey (1970).

\bibitem{T76} G. Talenti: \textit{Best constant in Sobolev inequality}.  Ann. Mat. Pura Appl. \textbf{110} (1976) 353-372.

\end{thebibliography}
\end{document}